\documentclass{amsart}[12pt]
\usepackage{mathrsfs, amsmath,amssymb}
\usepackage{fullpage}
\usepackage{mathtools}
\usepackage{xcolor}
\newtheorem{teo}{Theorem}
\newtheorem{pro}{Proposition}
\newtheorem{lem}{Lemma}
\newtheorem{cor}{Corollary}
\newtheorem*{rem}{Remark}

\title{Spectral properties of the free Jacobi process associated with equal rank projections}

\author{Nizar Demni}
\address{Department of Mathematics, New York University in Abu Dhabi, Saadiyat Island, P.O. Box 129188, Abu Dhabi, UAE}
\email{nd2889@nyu.edu}

\author[T. Hamdi]{Tarek Hamdi}
\address{Department of Management Information Systems  \\ College of Business and Economics\\ Qassim University  \\ Saudi Arabia
and Laboratoire d'Analyse Math\'ematiques et applications LR11ES11 \\ Universit\'e de Tunis El-Manar \\ Tunisie}
\email{ t.hamdi@qu.edu.sa } 


\begin{document}
\maketitle
\begin{abstract} 
The free Jacobi process is the radial part of the compression of the free unitary Brownian motion by two free orthogonal projections in  a non commutative probability space. In this paper, we derive spectral properties of the free Jacobi process associated with projections having the same rank $\alpha \in (0,1)$. To start with, we determine the characteristic curves of the partial differential equation satisfied by the moment generating function of its spectral distribution. 
Doing so leads for any fixed time $t >0$ to an expression of this function in a neighborhood of the origin, therefore extends our previous results valid for $\alpha = 1/2$. 
Moreover, the obtained characteristic curves are encoded by an $\alpha$-deformation of the compositional inverse of the $\chi$-transform of the spectral distribution of the free unitary Brownian motion.
In this respect, we study mapping properties of this deformation and use the saddle point method to prove that the compositional inverse of a $\alpha$-deformation of the $\chi$-transform of the free unitary Brownian motion is analytic in the open unit disc (for large enough time $t$).
The last part of the paper is devoted to a dynamical version of a recent identity pointed out by T. Kunisky in \cite{Kun}. Actually, this identity relates the stationary distributions of the free Jacobi processes corresponding to the sets of parameters 
$(\alpha, \alpha)$ and $(1/2,\alpha)$ respectively and we explain how it follows from the Nica-Speicher semi-group. Our dynamical version then relates the partial differential equations of the Cauchy-Stieltjes transforms of the densities of 
the finite-time spectral distributions. It also raises the problem of whether a dynamical analogue of the Nica-Speicher semi-group exists when the compressing projection has rank $1/2$.
\end{abstract}

%\tableofcontents

\section{Introduction and main results}
\subsection{Random matrices, matrix-valued stochastic processes and free probability theory} 
Free probability theory provides an operator-algebraic framework for large-size random matrices. It describes among other things the large-size weak limits of spectral distributions of independent sums and/or products of normal random matrices. 
As far as selfadjoint and unitary random matrices are considered, these limits are encoded in the so-called additive and multiplicative free convolutions on the real line and on the unit circle respectively. In this respect, 
complex analysis and combinatorial techniques open the way to compute their various transforms such as the Cauchy-Stieltjes and Herglotz transforms, or their compositional inverses referred to as R and S transforms, just to cite a few. 
We refer the reader to the monograph \cite{Nic-Spe} for a good background on free probability theory and to \cite{Min-Spe} for further related and advanced topics.   

Among the widely studied random matrix models figure independent Gaussian Hermitian and Haar unitary matrices. In the large size limit, they behave  as free semi-circular and free Haar unitary operators respectively in a $II_1$-type von Neumann algebra,
and their corresponding spectral distributions are given by the celebrated Wigner distribution and the Haar measure on the unit circle. An older and well-studied too selfadjoint model is given by complex Gaussian covariance or (known also as complex Wishart) matrices. When suitably rescaled, their spectral distributions converge weakly to the Marchencko-Pastur distribution. By considering the symmetrized ratio of two independent complex Wishart matrices, one obtains the complex matrix-variate Beta distribution which belongs to the so-called unitary Jacobi ensemble. In this respect, it was shown in \cite{Cap-Cas} that the corresponding spectral distribution converges weakly, when the ranks and the sizes of the underlying Wishart matrices have comparable growths, 
to the so-called Wachter distribution which already appeared in \cite{Wac}. Later on, another realization of the matrix-variate Beta distribution was afforded in \cite{Col} through radial parts of corners of Haar unitary matrices, which extends to higher dimensions the construction of the one dimensional Beta distribution from uniformly distributed vectors on Euclidean spheres. 

The study of matrix-valued stochastic processes originates in the work of F. Dyson (\cite{Dys}) and was motivated by interacting particle systems on the real line and on the circle. Indeed, the former consists of the eigenvalues of the Brownian motion in the space of Hermitian matrices while the latter consists of the eigenvalues of the Brownian motion in the unitary group. Motivated by principal component analysis, M. F. Bru introduced decades later the so-called Wishart processes whose complex Hermitian analogue was studied in \cite{DemLag}. The eigenvalues of this model then provide an interacting particle system on the positive half-line. However, the construction of the complex matrix-variate Beta distribution as a symmetrized ratio of two independent complex Wishart matrices does not extend to the dynamical setting. This is not even true for one dimensional Jacobi processes since the ratio of two independent Bessel processes is a time-changed Jacobi process (\cite{War-Yor}). 
Rather, the dynamical analogue of the Beta matrix model, known as the Hermitian Jacobi process, is defined as the radial part of an upper-left corner of a unitary Brownian matrix. Its eigenvalues form interacting particles in the interval $[0,1]$ 
(see \cite{Del-Dem} and references therein for further details). 

Coming to large-size limits of the aforementioned matrix-valued stochastic processes, the non commutative moments of the Hermitian Brownian motion converges to those of the free additive Brownian motion (\cite{BCG}). When starting at the zero operator, the spectral distribution of the latter is a time-rescaled Wigner distribution since the Hermitian Brownian motion at any time $t > 0$ reduces in this case to a Hermitian Gaussian matrix of variance $t$. Similar results hold true for the complex Wishart process: 
it converges (in the sense of non commutative moments) to the free Wishart process and when starting at zero, the spectral distribution of this self-adjoint operator at any time $t > 0$ is a time-rescaled Marchenko-Pastur distribution (\cite{Cap-Don}).   

As to the unitary Brownian motion and to the Hermitian Jacobi process, they converge in the large size limit  to the free unitary Brownian motion and to the free Jacobi process respectively (\cite{Bia}, \cite{Dem}). However, the uniform measure on the circle and the Wachter distribution only show up in the stationary regime $t \rightarrow +\infty$, while the spectral distributions of these unitary and Hermitian operators at any finite time $t > 0$ are completely different even when starting at the identity operator. Actually, the moments of the free unitary Brownian motion were derived in \cite{Bia} from the asymptotic analysis of the moments of the unitary Brownian motion and are expressed through Laguerre polynomials. 
As to those of free Jacobi process, they were recently obtained  in \cite{Dem-Ham1} after a careful asymptotic analysis of the moment formula derived in \cite{DHS} for the Hermitian Jacobi process.
%assuming comparable growth conditions for the parameters on which depend the Hermitian Jacobi process. 
%The main result proved in \cite{Dem-Ham1} provides for any integer $n \geq 1$ a doubly-indexed inductive relation whose top coefficient encodes the $n$-th moment of the free Jacobi process. 

Let us close this introductory part by recalling that the second named author thoroughly studied the dynamics of the spectral distribution of the free Jacobi process relying on the product of two orthogonal symmetries, one of which is rotated by a free unitary Brownian motion (\cite{Ham0}, \cite{Ham1}). This study revealed a close connection to radial L\"owner equations and led to a more or less explicit decomposition of the spectral distribution of the free Jacobi process at any time $t > 0$. In particular, this probability measure was shown to have only an atomic and an absolutely-continuous parts. 

\subsection{Contribution and main results}
Let $(\mathscr{A}, \tau)$ be a non commutative probability space endowed with a trace $\tau$ and with a unit ${\bf 1}$. Consider two orthogonal projections $P$ and $Q$ in $\mathscr{A}$ with ranks 
\begin{equation*}
\tau(P) = \beta, \quad \tau(Q) = \alpha, \quad \alpha, \beta  \in (0,1),
\end{equation*}
and let $(Y_t)_{t \geq 0}$ be a free unitary Brownian motion in $\mathscr{A}$. Assume $(Y_t)_{t \geq 0}$ is $\star$-free with $\{P, Q\}$ (\cite{Bia}), then the free Jacobi process $(J_t^{(\beta, \alpha)})_{t \geq 0}$ is defined for any time $t \geq 0$ by 
\begin{equation*}
J_t^{(\beta, \alpha)} := PY_tQY_t^{\star}P
\end{equation*}
viewed as a self-adjoint operator in the compressed algebra $P\mathscr{A}P$ endowed with the normalized trace $\tau/\tau(P)$. For the reader's convenience, let us point out that the ranks $(\beta, \alpha)$ correspond to $(\lambda \theta, \theta)$ in 
the notations used in \cite{Dem}, \cite{Del-Dem}, \cite{DHS} and \cite{Dem-Ham1}. This change of notations is only motivated by the connection to \cite{Kun} we explain later and we hope it will not bring any confusion.  

In this paper, we mainly focus on the case $\alpha = \beta$. At the matrix level, this couple of parameters corresponds to a Hermitian Jacobi process built out of an upper-left corner of a unitary Brownian motion whose shape parameters are 
asymptotically equal in the large size-limit (\cite{Del-Dem}). Now, let $\mu_t^{(\beta,\alpha)}$ denote the spectral distribution of $J_t^{(\beta,\alpha)}$ in the compressed probability space and denote further 
\begin{equation*}
M_t^{(\alpha)}(z):= \sum_{n \geq 0}m_{n,t}^{(\alpha)}z^n, \quad m_{n,t}^{(\alpha)}:= \int x^n \mu_t^{(\alpha,\alpha)}(dx),
\end{equation*}
the moment generating function of $\mu_t^{(\alpha,\alpha)}$. Since $0 < m_{n,t}^{(\alpha)} \leq 1$ then $M_t^{(\alpha)}$ is an analytic function in the open unit disc $\mathbb{D}$. Letting 
\begin{equation*}
G_t^{(\alpha)}(z) = \frac{1}{z}M_t^{(\alpha)}\left(\frac{1}{z}\right), \quad z \in \mathbb{C} \setminus \mathbb{D},
\end{equation*}
be the Cauchy-Stieltjes transform of $\mu_t^{(\alpha,\alpha)}$, then $G_t^{(\alpha)}$ extends to an analytic map to $\mathbb{C} \setminus [0,1]$ through the integral formula: 
\begin{equation*}
G_t^{(\alpha)}(z) = \int \frac{1}{z-x} \mu_t^{(\alpha,\alpha)}(dx).
\end{equation*}
 Besides, Proposition 7.1. in \cite{Dem} shows that it satisfies a transport-type partial differential equation (pde), and in turn so does $(t,z) \mapsto M_t^{(\alpha)}(z)$.
%Moreover, it satisfies a transport-type partial differential equation (pde) which may be readily deduced from  after a spatial variable change. Indeed, the aforementioned proposition provides a pde for the Cauchy-Stieltjes transform of $\mu_t^{(\alpha,\alpha)}$, say $G_t^{(\alpha)}$, which is related to $M_t^{(\alpha)}$ by 
The analysis of the characteristic curves of this pde leads to our first main result stated in the following theorem:

\begin{teo}\label{Th1}
For any $\alpha \in (0,1)$, set 
\begin{equation*}
A(\alpha): = \frac{1-2\alpha}{2\alpha}
\end{equation*}
and assume $\mu_0^{(\alpha,\alpha)} = \delta_1$. Then, for any time $t > 0$, there exists a map $J_{4\alpha t}^{(\alpha)}$ locally around the origin such that the moment generating function $M_t^{(\alpha)}$ is given by:
\begin{equation*}
M_t^{(\alpha)}(z) = \frac{-A(\alpha) +\sqrt{[A(\alpha)]^2z + (1-z)[J_{4\alpha t}^{(\alpha)}(z)]^2}}{1-z}.
\end{equation*}
Equivalently,
\begin{equation*}
M_t^{(\alpha)}(z) =J_{4\alpha t}^{(\alpha)}(z) \frac{1+\psi_t^{(\alpha)}(z)}{1-\psi_t^{(\alpha)}(z)} -A(\alpha),
\end{equation*}
where  
\begin{align*}
\psi_t^{(\alpha)}(z) :=   \frac{\left(J_{4\alpha t}^{(\alpha)}(z) - A(\alpha)-1\right)\left(J_{4\alpha t}^{(\alpha)}(z) - A(\alpha)\right)}{\left(J_{4\alpha t}^{(\alpha)}(z) + A(\alpha)+1\right)\left(J_{4\alpha t}^{(\alpha)}(z) + A(\alpha)\right)}e^{2\alpha  t J_{4\alpha t}^{(\alpha)}(z)}.
\end{align*}
\end{teo}
The map $J_{4\alpha t}^{(\alpha)}$ satisfies $J_{4\alpha t}^{(\alpha)}(0) = 1/(2\alpha)$ and arises as the local inverse in a neighborhood of $1/(2\alpha)$ of a $\alpha$-deformation of 
\begin{equation*}
u \mapsto \frac{4\xi_{2t}(u)}{(1+\xi_{2t}(u))^2} 
\end{equation*}
where 
\begin{equation*}
\xi_{2t}(u) := \frac{u-1}{u+1}e^{tu}
\end{equation*}
is the inverse of the Herglotz transform $H_{2t}$ of the spectral distribution of the free unitary Brownian motion at time $2t$ (\cite{Bia1}). In particular, if $\alpha = 1/2$ then $A(\alpha) = 0$ and 
\begin{equation*}
J_{2t}^{(1/2)}(z) = H_{2t}(\psi(z)), \quad \psi(z) :=\psi_t^{(1/2)}(z)= \frac{1-\sqrt{1-z}}{1+\sqrt{1-z}},
\end{equation*}
so that 
\begin{equation*}
M_t^{(1/2)}(z) = \frac{H_{2t}(\psi(z))}{\sqrt{1-z}}.
\end{equation*}
Since both sides of this equality are analytic in the open unit disc then it holds true there and as such one recovers Proposition 2 in \cite{DHH}. Even more, since $\psi$ is a analytic one-to-one map from the cup plane $\mathbb{C}\setminus [1,+\infty[$ 
onto $\mathbb{D}$ then the right-hand side gives an analytic extension of $M_t^{(1/2)}(z)$ to $\mathbb{C}\setminus [1,+\infty[$. On the other hand, it is worth noting that the proof of Proposition 2 in \cite{DHH} did not rely on the analysis of the characteristic 
curves of the corresponding pde, so that Theorem \ref{Th1} supplies another proof of it. 

Our analysis gives rise also to the following map: 
\begin{align*}
u \mapsto V_{4\alpha t}^{(\alpha)}(u) :=   \frac{\left(u - A(\alpha)-1\right)\left(u - A(\alpha)\right)}{\left(u + A(\alpha)+1\right)\left(u + A(\alpha)\right)}e^{2\alpha u t}
\end{align*}
which reduces to $\xi_{2t}$ when $\alpha = 1/2$ and $u \neq 0$. Recall from \cite{Bia1} that $\xi_{2t}$ is a analytic one-to-one map from a Jordan domain $\Gamma_{2t} \subset \{\Re(u) > 0\}$ onto $\mathbb{D}$ and that it exhibits a phase transition at $t=2$ illustrated by the fact that the closure $\overline{\Gamma_{2t}}$ intersects the imaginary axis when $t < 2$ while it lies totally inside the right half plane when $t \geq 2$. Consequently, the support of the spectral distribution of $Y_{2t}$ fits the whole unit circle 
only when $t \geq 4$. This phase transition carries over to the moments of $Y_{2t}$: they admit an exponential decay when $t > 2$ and a polynomial one when $t \leq 2$ (\cite{Gro-Mat}). 

It is then natural to check whether these properties extend to any $\alpha \in (0,1)$. Of course, one expects the failure of these properties for small times and small common rank $\alpha$ of both projections since they degenerate as $\alpha \rightarrow 0^+$. 
Nonetheless, since $Y_t$ approaches a Haar unitary operator when 
$t$ becomes large then we also expect a regularizing effect for large times even for smaller ranks. In this respect, we shall prove the following mapping properties on the positive half-line: 
\begin{itemize}
\item If $\alpha > 1/2$ then for any time $t > 0$, the rescaled map $u \mapsto V_{4\alpha t}^{(\alpha)}(u/(2\alpha))$ extends to a one-to-one map from an open interval $I_{t,\alpha} \subset (0,\infty)$ onto $(-1,1)$. 
\item If $\alpha < 1/2$ then for any time $t \in [0,2]$, the image of $(0,+\infty)$ by $u \mapsto V_{4\alpha t}^{(\alpha)}(u/(2\alpha))$ is a proper subset of $(-1,1)$. 
\end{itemize}  

The restriction to real arguments of $V_{4\alpha t}^{(\alpha)}$ is by no means a loss of generality. It is rather a sake of simplicity since the analysis of $V_{4\alpha t}^{(\alpha)}$ on the positive half line is already quite technical. 
Nonetheless, it shows phase transitions at two times $t_0(\alpha) \leq 2 \leq t_1(\alpha)$ provided that $\alpha > 1/2$ which collapse to $t=2$ when $\alpha = 1/2$. As to the decay of the moments of the local inverse (around the origin) of 
$V_{4\alpha t}^{(\alpha)}, \alpha > 1/2$, the saddle point analysis of its Taylor coefficients shows that there are four critical points as opposed to the two critical points corresponding to $\alpha = 1/2$ 
(see \cite{Gro-Mat} for the saddle point analysis in the latter case).
Moreover, an additional singularity at $u = -\alpha$ arises in the Cauchy integral representation of these Taylor coefficients when $\alpha > 1/2$ which eliminates the real critical points of the integrand lying outside the interval $(-\alpha, 0)$.  

The last result we prove in this paper is motivated by an identity due to T. Kunisky (\cite{Kun}, Proposition C.2.) stating that the normalized density of the spectral distribution of $\mu_{\infty}^{(\alpha,\alpha)}$ and the pushforward of the normalized density of 
$\mu_{\infty}^{(1/2,\alpha)}$ under the map $x \mapsto (2x-1)^2$ coincide. Here, we denoted $\mu_{\infty}^{(\beta, \alpha)}$ the Wachter distribution of parameters $(\beta, \alpha)$ which arises as the weak limit  of $\mu_{t}^{(\beta, \alpha)}$ as $t \rightarrow +\infty$. Though this identity may be readily checked from the explicit expressions of the densities of $\mu_{\infty}^{(\alpha,\alpha)}$ and of $\mu_{\infty}^{(1/2,\alpha)}$, we show that it is reminiscent of special instance of the  Nica-Speicher semi-group 
(\cite{Nic-Spe}): the compression of any operator by a free projection of rank $1/2$ is equivalent in distribution to the half-sum of two free copies of the given operator. 
Afterwards, we prove a dynamical version of Kunisky's identity which reads as follows: 
\begin{pro}\label{Pr1}
If $\tau(P) = 1/2$ and if the angle operators 
\begin{equation*}
PQP, \quad ({\bf 1}-P)Q({\bf 1}-P),
\end{equation*}
have the same spectral distribution in $(P\mathscr{A}P, 2\tau)$, then the Cauchy-Stieltjes transforms of the normalized density of $\mu_t^{(\alpha,\alpha)}$ and of the pushforward of the normalized density of $\mu_{t/2}^{(1/2,\alpha)}$ under the map 
$x \mapsto (2x-1)^2$ satisfy the same pde. 
\end{pro}
The proof of this proposition goes into two steps. The first one draws the same conclusion under the weaker assumption that the pushforward of the normalized density of $\mu_{t}^{(1/2,\alpha)}$ under the map $x \mapsto 2x-1$ 
is an even function. The second step shows that the assumption made in Proposition \ref{Pr1} on the angle operators ensures the validity of the weak assumption used in the first step.

Finally, we would like to emphasize that satisfying the same pde does not imply that the Cauchy-Stieltjes coincide unless their initial values at $t=0$ match. Nonetheless, we shall prove that if this coincidence holds true, then it implies equality 
between the even moments of $PQP$ in $(P\mathscr{A}P, 2\tau)$ and those of $(Q_1+Q_2)/2$ where $Q_1$ and $Q_2$ are two copies of $Q$. Though the latter hints to the freeness $P$ and $Q$ in 
which case it is an instance of the Nica-Speicher semi-group, we do not whether it holds beyond this setting. Said differently, it would be interesting to exhibit an example of non free orthogonal projections $P$ and $Q$ 
satisfying the assumption on the angle operators made in Proposition \ref{Pr1}. 

The paper is organized as follows. Section 2 is devoted to analysis of the characteristic curves of the pde satisfied by the moment generating function $M_t^{(\alpha)}$ and as such to the proof of Theorem \ref{Th1}. Section 3 is concerned with the mapping properties of the map $V_{4\alpha t}^{(\alpha)}$ and to the saddle point analysis of the Taylor coefficients of its local inverse when $t \geq t_1(\alpha)$. In the last section, we recall the identity pointed out by T. Kunisky and write another proof of it relying on the compression by a free projection and on an algebraic identity we proved in \cite{Dem-Ham}. There, we also prove Proposition \ref{Pr1}.

\section{Analysis of the Characteristic curves: proof of Theorem \ref{Th1}} 
For ease of reading, we split the proof of Theorem \ref{Th1} into three parts. In the first part, we obtain a Ricatti equation for the moment generating function $M_t^{(\alpha)}$ along the characteristic curves. In the second part, we solve locally this equation and derive the explicit expressions of the characteristic curves. In the last part, we use the local inversion Theorem to get the expression of $M_t^{(\alpha)}(z)$ displayed in Theorem \ref{Th1} locally around the origin. In particular, we specialize our findings to $\alpha = 1/2$ 
and show how to recover the expression of $M_t^{(1/2)}(z)$ derived in \cite{DHH}.

\subsection{A Ricatti equation}
Recall the the Cauchy-Stieltjes transform $G_t^{(\alpha)}$ of $\mu_{t}^{(\alpha, \alpha)}$:
\begin{equation*}
G_t^{(\alpha)}(z) = \int \frac{1}{z-x}\mu_{t}^{(\alpha, \alpha)}(dx) %= \sum_{n \geq 1}m_{n,t}^{(\alpha)}\frac{1}{z^{n+1}}, \quad z \in \mathbb{C} \setminus [0,1].
\end{equation*}
and recall from Proposition 7.1. in \cite{Dem} that: 
\begin{equation*}
\partial_tG_t^{(\alpha)}(z)= \partial_z\left[(1-2\alpha)zG_t^{(\alpha)}(z) +  \alpha z(z-1)[G_t^{(\alpha)}]^2(z) \right].
\end{equation*}
Rather than analyzing the characteristic curves of this pde, we shall consider the one satisfied by the moment generating function $M_t^{(\alpha)}$.
%\begin{equation*}
%M_t^{(\alpha, \alpha)}(z):= \frac{1}{z}G_t^{(\alpha, \alpha)}\left(\frac{1}{z}\right), 
%\end{equation*}
%which is an analytic function in the open unit disc $\mathbb{D}$. 
Then straightforward computations yield
\begin{equation*}
\partial_tM_t^{(\alpha)}(z)= -z \partial_z\left[(1-2\alpha)M_t^{(\alpha)}(z) +  \alpha (1-z)[M_t^{(\alpha)}]^2(z) \right], \quad |z| < 1 .
\end{equation*}
Consequently, the characteristic curves are differentiable curves $t \mapsto z_t^{(\alpha)} \in \mathbb{D}$ satisfying (at least) locally the non linear ordinary differential equation (ODE):
\begin{equation}\label{Charc1}
(z^{(\alpha)})'_t = (1-2\alpha)z_t^{(\alpha)} + 2\alpha z_t^{(\alpha)}(1-z_t^{(\alpha)}) f_t^{(\alpha)}, 
\end{equation}
where we set 
\begin{equation*}
f_t^{(\alpha)} := M_t^{(\alpha)}(z_t^{(\alpha)})
\end{equation*}
 and where the initial value $z_0$ lies in a neighborhood of the origin. Along these curves, the function $f^{(\alpha)}$ solves the following ODE:
\begin{equation}\label{Char2}
(f^{(\alpha)})'_t = \alpha z_t^{(\alpha)} [f_t^{(\alpha)}]^2.
\end{equation}
For sake of simplicity, let us choose $\mu_{0}^{(\alpha, \alpha)} = \delta_1$, that is:
\begin{equation*}
M_0^{(\alpha)}(z) = \frac{1}{1-z},
\end{equation*}
which implies that 
\begin{equation*}
f_0^{(\alpha)} = M_0^{(\alpha)}(z_0) = \frac{1}{1-z_0}.
\end{equation*}
From the operator-algebraic point of view, this choice amounts to take $P = Q$ so that $J_0^{(\alpha,\alpha)} = PY_0QY_0^{\star}P = P$. Nonetheless, the computations below obviously extend to any initial value $M_0^{(\alpha)}$ 
upon replacing $1-z_0$ by $1/M_0^{(\alpha)}(z_0)$.   

Back to the analysis of the characteristic curves, we define 
\begin{equation*}
y_t^{(\alpha)} := \int_0^t z_s^{(\alpha)} ds, 
\end{equation*}
and deduce from \eqref{Char2} that 
\begin{equation*}
f_t^{(\alpha)} = \frac{1}{1-z_0 -  \alpha y_t^{(\alpha)}}.
\end{equation*}
As a result, \eqref{Charc1} locally entails: 
\begin{equation*}
[1-z_0 - \alpha y_t^{(\alpha)}](y^{(\alpha)})^{''}_t = (1-2\alpha)[1-z_0 - \alpha y_t^{(\alpha)}](y^{(\alpha)})'_t + 2\alpha (y^{(\alpha)})'_t[1-(y^{(\alpha)})'_t],
\end{equation*}
%Setting further $w(t) = w(z_0,t) = 1-z_0 - \alpha y(t)$, we obtain  the following ODE
%\begin{align*}
%w(t)w^{''}(t) &= (1-2\alpha)w(t)w'(t) +  2w'(t)(\alpha+ w'(t)), \quad w(0) = 1-z_0, 
%\\& = w'(t) [2\alpha + (1-2\alpha)w(t)] + 2[w'(t)]^2.
%\end{align*}
whence we derive the autonomous ODE
\begin{equation*}
(f_t^{(\alpha)})^{''}_t = -(f^{(\alpha)})'_t [(2\alpha-1) - 2\alpha f_t^{(\alpha)}].
\end{equation*}
After integration, we get the Ricatti equation
 \begin{align*}
(f_t^{(\alpha)})'_t &= (1-2\alpha)f_t^{(\alpha)} + \alpha (f_t^{(\alpha)})^2 + C(z_0,\alpha),
\end{align*}
where $C(z_0,\alpha)$ is determined by the initial data at $t=0$:
\begin{align*}
C(z_0,\alpha) &= (f_t^{(\alpha)})'(0) + (2\alpha-1)f_t^{(\alpha)}(0) - \alpha [f_t^{(\alpha)}(0)]^2
\\& = \frac{\alpha z_0}{(1-z_0)^2} + \frac{2\alpha-1}{1-z_0} - \frac{\alpha}{(1-z_0)^2} = \frac{\alpha-1}{1-z_0}.
\end{align*}
Completing the square, the Ricatti equation takes the form 
 \begin{align}\label{Ric0}
(f_t^{(\alpha)})'_t  = \alpha \left(f_t^{(\alpha)} +\frac{1-2\alpha}{2\alpha}\right)^2 - \frac{1-\alpha}{1-z_0} - \frac{(1-2\alpha)^2}{4\alpha},
\end{align}
and letting $g_t^{(\alpha)} = f_{t/\alpha}^{(\alpha)}$, it readily follows that 
 \begin{align}\label{Ric}
(g_t^{(\alpha)})'_t= \left(g_t^{(\alpha)}+A(\alpha)\right)^2 -B(1-z_0,\alpha),
\end{align}
where 
\begin{equation}\label{DefAB}
A(\alpha) = \frac{1-2\alpha}{2\alpha}, \quad B(1-z_0, \alpha) := \frac{1-\alpha}{\alpha(1-z_0)} + [A(\alpha)]^2.
\end{equation}

\subsection{Back to the case $\alpha =1/2$}
As we already mentioned in the introduction, the Lebesgue decomposition of $\mu_{t}^{(1/2, 1/2)}$ is explicit and already known. However, its derivation in \cite{DHH} did not rely on characteristic curves of the pde satisfied by $M_t^{(1/2)}$. In this paragraph, 
we specialize the findings of the previous paragraph to $\alpha = 1/2$ and retrieve $M_t^{(1/2)}$. To this end, recall from \cite{DHH} the expression:
\begin{equation}\label{EXP12}
M_t^{(1/2)}(z)= \frac{1}{\sqrt{1-z}} \left[1+2\eta_{2t}(\psi(z))\right], \quad |z| < 1,
\end{equation}
where 
\begin{equation*}
\psi(z) := \frac{1-\sqrt{1-z}}{1+\sqrt{1-z}}, \quad z \in \mathbb{C} \setminus [1, +\infty[,
\end{equation*}
and 
\begin{equation*}
\eta_{2t}(z) := \sum_{n \geq 1} \tau[(Y_{2t})^n]z^n,
\end{equation*}
is the moment generating function of the spectral distribution of $Y_{2t}$. Note in passing that $\psi$ is the compositional inverse in the cut plane $\mathbb{C} \setminus [1, +\infty[$ onto the open unit disc $\mathbb{D}$ 
of the map 
\begin{equation*}
\psi^{-1}(z) = \frac{4z}{(1+z)^2},
\end{equation*} 
and that 
\begin{equation*}
H_{2t}(z) := 1 + 2\eta_{2t}(z)
\end{equation*} 
is the Herglotz transform of the spectral distribution of $Y_{2t}$ (\cite{Bia1}). Now, if $\alpha = 1/2$ then $A(1/2) = 0$ and 
\begin{equation*}
B(1-z_0, 1/2) = \frac{1}{1-z_0} = M_0^{(1/2)}(z). 
\end{equation*}
Consequently, the Ricatti equation \eqref{Ric} reduces to 
\begin{equation*}
(g^{(1/2)})'_t = (g_t^{(1/2)})^2 - B(1-z_0, 1/2),
\end{equation*}
which yields after integration:
\begin{equation*}
f_t^{(1/2)} = g_{t/2}^{(1/2)} = \sqrt{B(1-z_0, 1/2)} \frac{1+\xi_{2t}(\sqrt{B(1-z_0, 1/2)})}{1-\xi_{2t}(\sqrt{B(1-z_0, 1/2)})}
\end{equation*}
where we recall the map:
\begin{equation*}
\xi_{2t}(z) = \frac{z-1}{z+1}e^{tz}.
\end{equation*}
%Actually, $\xi_{2t}$ is the compositional inverse of $1+ 2\rho_{2t}$ in a neighborhood of $z=1$ lying in the open right half plane. 
Keeping in mind \eqref{Char2} and using again \eqref{Ric0}, we further obtain 
\begin{align*}
z_t^{(1/2)} = \frac{2(f^{(1/2)})'_t}{[f_t^{(1/2)}]^2}  & = 1-\frac{B(1-z_0, 1/2)}{[f_t^{(1/2)}]^2} 
\\& = \frac{4\xi_{2t}(\sqrt{B(1-z_0,1/2)})}{(1+\xi_{2t}(\sqrt{B(1-z_0,1/2)}))^2} = \psi^{-1}\left[\xi_{2t}(\sqrt{B(1-z_0,1/2)})\right]. 
\end{align*}
Equivalently, $\xi_{2t}\left[\sqrt{B(1-z_0, 1/2)}\right] = \psi(z_t^{(1/2)})$, whence $\sqrt{B(1-z_0,1/2)} = 1+2\eta_{2t}[\psi(z_t^{(1/2)}))]$ since $\xi_{2t}$ is the compositional inverse of $H_{2t}$ in $\mathbb{D}$. Altogether, we end up with: 
\begin{equation*}
M_t^{(1/2)}(z_t) = f_t^{(1/2)} = \sqrt{B(1-z_0,1/2)} \frac{1+\psi(z_t^{(1/2)} )}{1-\psi(z_t^{(1/2)} )} = \frac{1}{\sqrt{1-z_t^{(1/2)}}} \left[1+2\eta_{2t}(\psi(z_t^{(1/2)} ))\right],
\end{equation*}
and retrieve \eqref{EXP12} locally around the origin. By analytic continuation, \eqref{EXP12} extends to $\mathbb{D}$ and its RHS is analytic in $\mathbb{C} \setminus [1, +\infty[$.

\subsection{General ranks} 
Mimicking the previous computations, we integrate locally the Ricatti equation \eqref{Ric} for real $z_0$. Before proceeding, note firstly that 
\begin{equation*}
\left(g_0^{(\alpha)}+A(\alpha)\right)^2 -B(1-z_0,\alpha) = \frac{z_0}{(1-z_0)^2}
\end{equation*}
does not vanish unless $z_0 = 0$ in which case \eqref{Charc1} reduces simply to $z^{(\alpha)} = 0$ locally. Secondly, for fixed and small enough $z_0 > 0$\footnote{The analysis is similar for $z < 0$ and leads to the same result.}, \eqref{Ric} is integrated as:
\begin{equation*}
\int_{g_0^{(\alpha)}}^{g_t^{(\alpha)}} \frac{du}{\left(u+A(\alpha)\right)^2 -B(1-z_0,\alpha)} = t,
\end{equation*} 
or equivalently 
\begin{equation}\label{E0}
f_t^{(\alpha)} = g_{\alpha t}^{(\alpha)} = \sqrt{B(1-z_0,\alpha)}\frac{1+F_{4\alpha t}^{(\alpha)}(1-z_0)}{1-F_{4\alpha t}^{(\alpha)}(1-z_0)} - A(\alpha),
\end{equation}
where we set 
\begin{equation*}
F_{4\alpha t}^{(\alpha)}(1-z_0) := \frac{\mathcal{I}(1-z_0,\alpha)-1}{\mathcal{I}(1-z_0,\alpha)+1}e^{2\alpha \sqrt{B(1-z_0,\alpha)} t}
\end{equation*}
and %(for $z_0$ sufficiently close to the origin)
\begin{equation*}
\mathcal{I}(1-z_0,\alpha) := \frac{f(0)+ A(\alpha)}{\sqrt{B(1-z_0,\alpha)}} = \frac{A(\alpha)+ \displaystyle \frac{1}{1-z_0}}{\sqrt{[A(\alpha)]^2 + \displaystyle \frac{1-\alpha}{\alpha(1-z_0)}}}.
%= \frac{1+A(\alpha)(1-z_0)}{\sqrt{[A(\alpha)]^2(1-z_0)^2 + \displaystyle \frac{1-\alpha}{\alpha}(1-z_0)}}.
\end{equation*}
Now, we appeal to \eqref{Char2} together with \eqref{Ric0} to write 
\begin{equation}\label{E1}
z_t^{(\alpha)} = \frac{(f_t^{(\alpha)})'_t}{\alpha [f_t^{(\alpha)}]^2}  = 1 + \frac{2A(\alpha)}{f_t^{(\alpha)}} - \frac{B(1-z_0,\alpha) - [A(\alpha)]^2}{[f_t^{(\alpha)}]^2}.
\end{equation}
Consequently, we have 
\begin{equation}\label{E2}
f_t^{(\alpha)} = \frac{-A(\alpha) +\sqrt{[A(\alpha)]^2z_t^{(\alpha)} + B(1-z_0,\alpha)(1-z_t^{(\alpha)})}}{1-z_t^{(\alpha)}} 
\end{equation}
where the choice of the square root branch is subject to the initial value $f(0) = M_0^{(\alpha)}(z_0) = 1/(1-z_0)$. 

Note that this expression of $f_t^{(\alpha)} = M_t^{(\alpha)}(z_t^{(\alpha)})$ has the same shape as the one of the moment generating function of the stationary spectral distribution of $J_t$ (\cite{Dem}, p.129):
\begin{align*}
M_{\infty}^{(\alpha)}(z) := \lim_{t \rightarrow + \infty} M_t^{(\alpha)}(z) &= \frac{(2\alpha-1)/\alpha +\sqrt{1/\alpha^2 - 4(1-\alpha)z/\alpha}}{2(1-z)} 
\\& = \frac{-A(\alpha) +\sqrt{[A(\alpha)]^2z + (1-z)/(4\alpha^2)}}{1-z} 
\\& = \frac{-A(\alpha) +\sqrt{[A(\alpha)]^2z + [A(\alpha) +1]^2(1-z)}}{1-z} 
\end{align*}
where we recall that $A(\alpha) = (1-2\alpha)/(2\alpha)$.

Finally, it remains to express $\sqrt{B(1-z_0, \alpha)}$ through $z_t^{(\alpha)}$ for fixed $\alpha$. To this end, we find it better to rewrite (using the definition of $B(1-z_0,\alpha)$ in equation \eqref{DefAB}): 
\begin{align*}
\mathcal{I}(1-z_0,\alpha) & = \frac{\alpha}{1-\alpha} \frac{B(1-z_0,\alpha) + \displaystyle \frac{1-2\alpha}{4\alpha^2}}{\sqrt{B(1-z_0,\alpha)}}
\\& = \frac{\alpha}{1-\alpha} \left[\sqrt{B(1-z_0,\alpha)} +  \frac{1-2\alpha}{4\alpha^2 \sqrt{B(1-z_0,\alpha)}}\right].
\end{align*}
Doing so allows to express $F_{4\alpha t}^{(\alpha)}(1-z_0)$ (and in turn $f(t)$) only through $u:= \sqrt{B(1-z_0,\alpha)}$ as:
\begin{align*}
F_{4\alpha t}^{(\alpha)}(1-z_0) & = \frac{\sqrt{B(1-z_0,\alpha)} +  \displaystyle \frac{1-2\alpha}{4\alpha^2 \sqrt{B(1-z_0,\alpha)}} - \frac{1-\alpha}{\alpha}}{\sqrt{B(1-z_0,\alpha)} +  \displaystyle \frac{1-2\alpha}{4\alpha^2 \sqrt{B(1-z_0,\alpha)}} + \frac{1-\alpha}{\alpha}}e^{2\alpha \sqrt{B(1-z_0,\alpha)} t} 
\\& = V_{4\alpha t}^{(\alpha)}(\sqrt{B(1-z_0,\alpha)})
\end{align*}
where
\begin{align*}
V_{4\alpha t}^{(\alpha)}(u) := \frac{u+  \displaystyle \frac{1-2\alpha}{4\alpha^2 u} - \frac{1-\alpha}{\alpha}}{u +  \displaystyle \frac{1-2\alpha}{4\alpha^2 u} + \frac{1-\alpha}{\alpha}}e^{2\alpha u t}
& = \frac{\left(u - \displaystyle \frac{1}{2\alpha}\right)\left(u - \displaystyle \frac{1-2\alpha}{2\alpha}\right)}{\left(u + \displaystyle \frac{1}{2\alpha}\right)\left(u + \displaystyle \frac{1-2\alpha}{2\alpha}\right)}e^{2\alpha u t}
\\& =  \frac{\left(u - A(\alpha)-1\right)\left(u - A(\alpha)\right)}{\left(u + A(\alpha)+1\right)\left(u + A(\alpha)\right)}e^{2\alpha u t}.
\end{align*}
Consequently,  \eqref{E1} may be written as 
\begin{equation}\label{Inversion}
z_t^{(\alpha)} = 1+ 2A(\alpha) \frac{1-V_{4\alpha t}^{(\alpha)}(u)}{(u-A(\alpha)) +(u+A(\alpha))V_{4\alpha t}^{(\alpha)}(u)} - \left[u^2-(A(\alpha))^2\right] \left[\frac{1-V_{4\alpha t}^{(\alpha)}(u)}{(u-A(\alpha)) +(u+A(\alpha))V_{4\alpha t}^{(\alpha)}(u)} \right]^2   
\end{equation}
 in the variable $u = \sqrt{B(1-z_0,\alpha)}$. 
 
 \subsection{Local invertibility: end of the proof of Theorem \ref{Th1}} 
 By the virtue of \eqref{E2}, we can get an expression of $M_t(z)$ locally around zero provided that we locally invert \eqref{Inversion} in the variable $u$ near $1/(2\alpha) = \sqrt{B(1, \alpha)}$. 
 To this end, we need to compute the derivative of the RHS of \eqref{Inversion} at $1/(2\alpha)$.  In this respect, quick computations yield  
 \begin{equation*}
 \partial_u V_t^{(\alpha)}\left(\frac{1}{2\alpha}\right) = e^t \frac{\alpha^2}{1-\alpha},
 \end{equation*}
 whence we deduce that the sought derivative equals to:
 \begin{equation*}
\frac{(1+2\alpha)}{2(1-\alpha)}e^t > 0.
 \end{equation*}
Consequently, there exists a function $J_{4\alpha t}^{(\alpha)}$ locally around the origin such that 
\begin{equation*}
J_{4\alpha t}^{(\alpha)}(0) = \frac{1}{2\alpha}, \quad \sqrt{B(1-z_0, \alpha)} = J_{4\alpha t}^{(\alpha)}(z_t^{(\alpha)}), 
\end{equation*}
and it follows from \eqref{E2} that 
\begin{equation*}
f_t^{(\alpha)} = \frac{-A(\alpha) +\sqrt{[A(\alpha)]^2z(t) + (1-z(t))[J_{4\alpha t}^{(\alpha)}(z(t))]^2}}{1-z(t)}.
\end{equation*}
As a result,
\begin{equation*}
M_t^{(\alpha)}(z) = \frac{-A(\alpha) +\sqrt{[A(\alpha)]^2z + (1-z)[J_{4\alpha t}^{(\alpha)}(z)]^2}}{1-z}
\end{equation*}
locally around the origin, proving the first part of Theorem \ref{Th1}.

Finally, let $\psi_{t}^{(\alpha)}(z) :=V_{4\alpha t}^{(\alpha)}\left[J_{4\alpha t}^{(\alpha)}(z)\right]$, then \eqref{E0} reads
\begin{align*}
f_t^{(\alpha)} &= \sqrt{B(1-z_0,\alpha)}\frac{1+F_{4\alpha t}^{(\alpha)}(1-z_0)}{1-F_{4\alpha t}^{(\alpha)}(1-z_0)} - A(\alpha) 
\\&= J_{4\alpha t}^{(\alpha)}(z_t^{(\alpha)}) \frac{1+V_{4\alpha t}^{(\alpha)}\left[J_{4\alpha t}^{(\alpha)}(z_t^{(\alpha)})\right]}{1-V_{4\alpha t}^{(\alpha)}\left[J_{4\alpha t}^{(\alpha)}(z_t^{(\alpha)})\right]} - A(\alpha)
\\& = J_{4\alpha t}^{(\alpha)}(z_t^{(\alpha)}) \frac{1+\psi_{4\alpha t}^{(\alpha)}(z_t^{(\alpha)})}{1-\psi_{4\alpha t}^{(\alpha)}(z_t^{(\alpha)})} - A(\alpha).
\end{align*} 
Then, it holds locally around the origin that
\begin{align}\label{Moment-generating2}
M_t^{(\alpha)}(z) = J_{4\alpha t}^{(\alpha)}(z)\frac{1+\psi_t^{(\alpha)}(z)}{1-\psi_t^{(\alpha)}(z)}-A(\alpha)
\end{align}
and Theorem \ref{Th1} is proved.

\section{Properties of the map $V_{4\alpha t}^{(\alpha)}$}
When $\alpha = 1/2$, the map $V_{2t}^{(1/2)}$ reduces on the open right half plane to $\xi_{2t}$, which is known to be the compositional inverse in a Jordan domain $\Gamma_t$ 
containing $z=1$ onto $\mathbb{D}$ of the Herglotz transform $H_{2t} = 1 + 2\eta_{2t}$ of the spectral distribution of $Y_{2t}$ (\cite{Bia1}). 
Equivalently, the map 
\begin{equation*}
u \mapsto \frac{u}{1+u}e^{t(1+2u)} = \xi_{2t}(1+2u)
\end{equation*}
is the inverse of the moment generating function $\eta_{2t}$ in a neighborhood of the origin onto $\mathbb{D}$ and in turn $u \mapsto e^{t(1+2u)}$ is the $S$-transform of $Y_{2t}$ (\cite{Ber-Voi}). 
It is then natural to wonder whether these properties extend or not to any $\alpha \neq 1/2$. Regarding the first property, we shall investigate the variations of $V_{4\alpha t}^{(\alpha)}$ in the positive half-line and in particular to check whether it is a bijection from some interval containing $z=1$ onto $(-1,1)$. Though this is a technical (and at some point tricky) exercise from real analysis, it already reveals on the one hand the contrast between the values 
$\alpha <1/2$ and $\alpha > 1/2$ as well as the interplay between the rank $\alpha$ and the time $t$ when $\alpha > 1/2$. On the other hand, 
our computations below show that the phase transition occurring at time $t=2$ when $\alpha = 1/2$ and corresponding to the free unitary Brownian motion splits into two phase transitions when $\alpha > 1/2$ at two times $t_0(\alpha) \leq 2 \leq t_1(\alpha)$ 
which collapse when $\alpha = 1/2$.  

\subsection{Variations of $V_{4\alpha t}^{(\alpha)}$ on the real line}
For sake of simplicity, let us consider the rescaled map 
\begin{equation*}
\tilde{V}_{t}^{(\alpha)}(u) := V_{4\alpha t}^{(\alpha)}\left(\frac{u}{2\alpha}\right) = \frac{(u-1)(u-1+2\alpha)}{(u+1)(u+1-2\alpha)} e^{tu}, \quad u > 0,
\end{equation*}
which is singular at $2\alpha-1$ if $\alpha > 1/2$ and vanishes otherwise at $1-2\alpha \geq 0$. In this respect, we shall prove: 
\begin{pro}
If $\alpha \geq 1/2$ then for any $t > 0$, $\tilde{V}_{t}^{(\alpha)}$ is a bijection from an open interval 
\begin{equation*}
I_{t,\alpha} \subsetneq (2\alpha-1, +\infty), 
\end{equation*}
onto $(-1,1)$. Otherwise 
\begin{equation*}
\tilde{V}_{t}^{(\alpha)}\left([0,+\infty)\right) \subsetneq (-1,1)
\end{equation*}
for any time $t \in [0,2]$. 
\end{pro} 

\begin{proof}
Straightforward computations show that the derivative of $\tilde{V}_{4\alpha t}^{(\alpha)}$ is given by: 
\begin{equation*}
\partial_u \tilde{V}_{t}^{(\alpha)}(u) = \frac{e^{ut}}{(u+1)^2(u+1-2\alpha)^2}\left[4(1-\alpha)(u^2-1+2\alpha) +t(u^2-1)(u^2 - (2\alpha-1)^2\right].
\end{equation*}
Set 
\begin{equation*}
R_{t,\alpha}(y) := 4(1-\alpha)(y-1+2\alpha) +t(y-1)(y - (2\alpha-1)^2), \quad y > 0,
\end{equation*}
so that 
\begin{equation*}
\partial_u \tilde{V}_{t}^{(\alpha)}(u) = \frac{e^{ut}}{(u+1)^2(u+1-2\alpha)^2}R_{t,\alpha}(u^2).
\end{equation*}
Then 
\begin{equation*}
\partial_y R_{t,\alpha}(y) = 4(1-\alpha) +t[2y - (1+(1-2\alpha)^2)] 
\end{equation*}
which vanishes if and only if 
\begin{equation*}
t > T(\alpha) := \frac{4(1-\alpha)}{1+(1-2\alpha)^2}. 
\end{equation*}
We are then led to distinguish two cases: 
\begin{enumerate}
\item $t \leq T(\alpha)$: in this case $\partial_yR_{t,\alpha}\geq 0$ and $R_{t,\alpha}$ has at most one real root $y_{t, \alpha}$. 
%Equivalently, $\partial_u \tilde{V}_{t}^{(\alpha)}(\sqrt{y_{t,\alpha}}) = 0$ and we further distinguish two subcases: 
\begin{itemize}
\item If $\alpha > 1/2$ then $R_{t,\alpha}(0) > 0$ and in turn $\tilde{V}_{t}^{(\alpha)}$ is increasing on $(2\alpha-1, +\infty)$. It is therefore a bijection from this interval onto $\mathbb{R}$ and there exists a unique open interval 
$I_{t,\alpha} \subset (2\alpha-1, +\infty)$ whose image is exactly $(-1,1)$. Moreover, $I_{t,\alpha} \subset (\sqrt{2\alpha-1}, +\infty)$ since 
\begin{equation*}
\tilde{V}_{t}^{(\alpha)}(\sqrt{2\alpha-1}) = -e^{t\sqrt{2\alpha-1}} < -1.
\end{equation*}
\item Otherwise $\alpha < 1/2$ and we can readily check that $T(\alpha) > 2$ and that 
\begin{equation*}
R_{t,\alpha}(0) < 0. 
\end{equation*}
Consequently, $R_{t,\alpha}$ has a unique root $y_{t,\alpha}$ which further satisfies: 
\begin{equation*}
\sqrt{1-2\alpha} \leq \sqrt{y_{t,\alpha}} < 1
\end{equation*}
 since 
\begin{equation*}
R_{t,\alpha}(1-2\alpha) = -4\alpha^2t (1-2\alpha) \leq 0 < R_{t,\alpha}(1) = 8\alpha(1-\alpha).
\end{equation*}  
Equivalently, $\sqrt{y_{t,\alpha}}$ is the unique minimum of $\tilde{V}_{t}^{(\alpha)}$ in $(0,\infty)$ and we will show that 
$$\tilde{V}_{t}^{(\alpha)}(\sqrt{y_{t,\alpha}}) > -1$$ 
for any $t \in [0,2]$. For these times, we shall use the fact that the map 
$u \mapsto \xi_{2t}(u)$ is increasing on $(0,\infty)$ (\cite{Bia1}) in order to derive the estimate
\begin{equation*}
|\tilde{V}_{t}^{(\alpha)}(u)| = |\xi_{t}(u)|  \frac{(u-1+2\alpha)}{(u+1-2\alpha)} \leq \frac{\alpha}{1-\alpha} < 1,
\end{equation*}
valid for any $\sqrt{2\alpha-1} \leq u < 1$ whence $\tilde{V}_{t}^{(\alpha)}(\sqrt{y_{t,\alpha}}) > -1, t \in [0,2],$ as desired. Since 
\begin{equation*}
\tilde{V}_{t}^{(\alpha)}(0) = 1, \quad \lim_{u \rightarrow +\infty} \tilde{V}_{t}^{(\alpha)}(u) = +\infty,
\end{equation*}
\end{itemize}
then we conclude that $\tilde{V}_{t}^{(\alpha)} \left([0,\infty)\right) \subsetneq (-1,1)$ for any $t \in [0,2]$. 

\item $t > T(\alpha)$: 
in this case, the unique root of $\partial_y R_{t,\alpha}$ is 
\begin{equation*}
\frac{t[1+(1-2\alpha)^2] - 4(1-\alpha)}{2t} < 1
\end{equation*}
and we will only consider values $\alpha > 1/2$. In this respect, we compute the minimum value of $R_{t,\alpha}$:
\begin{equation*}
R_{t,\alpha}\left(\frac{t[1+(1-2\alpha)^2] - 4(1-\alpha)}{2t} \right)  = \frac{4(1-\alpha)}{t}[2\alpha^2 t-\alpha^2(1-\alpha)t^2 - (1-\alpha)],
\end{equation*}
which takes positive values on $[t_0(\alpha), t_1(\alpha)]$ and negative values when $t \geq t_1(\alpha)$, where 
\begin{equation}\label{Times}
t_0(\alpha):= \frac{\alpha - \sqrt{2\alpha-1}}{\alpha(1-\alpha)} \leq T(\alpha) \leq  t_1(\alpha):= \frac{\alpha + \sqrt{2\alpha-1}}{\alpha(1-\alpha)}.
\end{equation}
If $t \in [t_0(\alpha), t_1(\alpha)]$ then $\tilde{V}_{t}^{(\alpha)}$ is increasing on $(2\alpha-1, +\infty)$ onto $\mathbb{R}$ and we draw the same conclusion about the existence of $I_{t,\alpha}$. 

Otherwise, $t > t_1(\alpha)$ whence we deduce that 
\begin{align*}
\frac{t[1+(1-2\alpha)^2] - 4(1-\alpha)}{2t} > \frac{1+(2\alpha-1)^2}{2} -\frac{2(1-\alpha)}{t_1(\alpha)} & = \sqrt{2\alpha-1}(2\alpha - \sqrt{2\alpha-1}) > 2\alpha-1.
%\\&  > (2\alpha-1)(2\alpha - \sqrt{2\alpha-1})^2  > (2\alpha-1)^2.
\end{align*}
Moreover, straightforward computations show that
\begin{equation*}
R_{t,\alpha}\left(\sqrt{2\alpha-1}(2\alpha - \sqrt{2\alpha-1})\right) = 8\alpha\sqrt{2\alpha-1}\left[1-\alpha - \alpha t(\alpha- \sqrt{2\alpha-1})\right] < 0.
\end{equation*}
As a result, $\tilde{V}_{t}^{(\alpha)}$ has a unique minimum in the interval 
\begin{equation*}
\left[\sqrt{\sqrt{2\alpha-1}(2\alpha - \sqrt{2\alpha-1})}, +\infty\right).
\end{equation*}
It then suffices to prove that the corresponding minimum value is less than $-1$. To this end and for sake of simplicity, notice that 
\begin{equation*}
\sqrt{2\alpha-1}(2\alpha - \sqrt{2\alpha-1}) < \sqrt{2\alpha-1}
\end{equation*}
and that 
\begin{align*}
\tilde{V}_{t}^{(\alpha)}[(2\alpha-1)^{1/4}] & = -\frac{1-(2\alpha-1)^{1/4}}{1+(2\alpha-1)^{1/4}} \frac{1+(2\alpha-1)^{3/4}}{1-(2\alpha-1)^{3/4}}e^{t(2\alpha-1)^{1/4}} 
\\& = - \frac{x^2-x+1}{x^2+x+1}e^{tx}{}_{|x = (2\alpha-1)^{1/4}}.
\end{align*}
Since $t > t_1(\alpha) \geq 2$, then 
\begin{equation*}
x\mapsto  -\frac{x^2-x+1}{x^2+x+1}e^{tx}, 
\end{equation*}
is decreasing on $[0,1]$ so that 
\begin{equation*}
\tilde{V}_{t}^{(\alpha)}[(2\alpha-1)^{1/4}] < -1 
\end{equation*}
as desired. The proposition is proved. 
\end{enumerate}
\end{proof}

\begin{rem}
Recall from the previous proof that if $\alpha < 1/2$ and $t \leq T(\alpha)$, then $\sqrt{y_{t,\alpha}} \in (\sqrt{1-2\alpha}, 1)$ is the unique minimum of $\tilde{V}_{t}^{(\alpha)}$ in $(0,\infty)$ with minimal value $\tilde{V}_{t}^{(\alpha)}(\sqrt{y_{t,\alpha}}) > -1$. 
However, note that 
\begin{equation*}
\tilde{V}_{t}^{(\alpha)}(\sqrt{1-2\alpha}) = -\frac{(1-\sqrt{1-2\alpha})^2}{(1+\sqrt{1-2\alpha})^2}e^{t\sqrt{1-2\alpha})} < -1 
\end{equation*}
for large enough time $t$.
\end{rem}

\begin{rem}
The previous proof shows that if $\alpha \geq 1/2$ then the inverse image of $-1$ in $I_{t,\alpha}$ is bigger than $\sqrt{2\alpha-1}$ for any time $t > 0$ and that it is further bigger than 
\begin{equation*}
\sqrt{\sqrt{2\alpha-1}(2\alpha - \sqrt{2\alpha-1})} > \sqrt{2\alpha-1}
\end{equation*}
when $t \geq t_1(\alpha)$. %This last lower bound allows to extend the computations done in \cite{Bia}, p. 268, to any $\alpha \in [1/2,1)$. 
However, numerical evidences show that 
\begin{equation*}
\tilde{V}_{t}^{(\alpha)}\left(\sqrt{\sqrt{2\alpha-1}(2\alpha - \sqrt{2\alpha-1})}\right) > -1
\end{equation*}
for $t \in [0,2)$ and small values of $\alpha \geq 1/2$.
\end{rem}

\section{A deformed $\chi$-transform: a saddle point analysis}
The map $\tilde{V}_{t}^{(\alpha)}$ reduces when $\alpha = 1/2$ to $\tilde{V}_{t}^{(1/2)} = \xi_{2t}$. Performing the variable change $u \mapsto 1+2u$ there and multiplying afterwards by $(u+1)/u$ lead to:   
\begin{equation*}
S_t^{(\alpha)}: u \mapsto  \frac{u + \alpha}{u + 1 - \alpha}e^{(1+2u) t},
 \end{equation*}
 which is obviously a $\alpha$-deformation of the S-transform of the free unitary Brownian motion $Y_{2t}$ (\cite{Ber-Voi}): 
 \begin{equation*}
 S_t^{(1/2)} := S_{Y_{2t}} := e^{(1+2u) t}. 
 \end{equation*} 
  Consequently, it is natural to ask whether $S_t^{(\alpha)}, \alpha \in (0,1],$ is still the S-transform of a probability distribution supported in the unit circle for any time $t \geq 0$.
 In this respect, note that Proposition 2.2 in \cite{Ber-Voi} forces the condition 
\begin{equation*}
|S_t^{(\alpha)}(0)| \geq  1 \quad \Leftrightarrow \quad \alpha \geq 1-\alpha \quad  \Leftrightarrow  \quad \alpha \geq 1/2.
\end{equation*}
However, note that the map
 \begin{equation*}
u \mapsto \frac{(u + \alpha)}{(u + 1 - \alpha)}
 \end{equation*}
is not the $S$-transform of a probability distribution on the unit circle unless $\alpha = 1/2$, and in turn $S_t^{(\alpha)}$ is not the product of two $S$-transforms. To see this, consider the following map:
\begin{equation}\label{STUC}
u \mapsto \frac{u(u + \alpha)}{(u+1)(u + 1 - \alpha)}
\end{equation}
and perform the variable change $u = z/(1-z)$ in a neighborhood of the origin to  
\begin{equation}\label{TSTUC}
z \mapsto z \frac{(1-\alpha) z + \alpha}{\alpha z + (1-\alpha)}, \quad |z| < (1-\alpha)/\alpha \leq 1.
\end{equation}
Assuming \eqref{STUC} is a $S$-transform of a probability distribution on the unit circle, then the map displayed in \eqref{TSTUC} would be the compositional inverse of the so-called $\eta$-transform 
of a probability measure on the unit circle (\cite{Bel-Ber}, p.73). In particular, this inverse would be analytic in the open unit disc $\mathbb{D}$ which is not true unless $\alpha = 1/2$. Indeed, the latter is given locally by:
\begin{equation*}
z \mapsto \frac{\sqrt{\alpha^2(1-z)^2 + 4(1-\alpha)^2 z} - \alpha(1-z)}{2(1-\alpha)},
\end{equation*}
which does not extend to an analytic function in $\mathbb{D}$ for any $\alpha > 1/2$ since 
\begin{equation*}
z \mapsto \alpha^2(1-z)^2 + 4(1-\alpha)^2 z
\end{equation*}
may take negative values on the vertical line $\{ x = (\alpha^2-2(1-\alpha)^2)/\alpha^2\}$ lying inside the open unit disc. 

All that to say that we are led to study the analyticity of the local inverse near the origin of:
 \begin{equation*}
\chi_t^{(\alpha)}(u) :=  \frac{u(u + \alpha)}{(u +1)(u + 1 - \alpha)}e^{(1+2u) t}
 \end{equation*}
which reduces when $\alpha =1/2$ to the $\chi$-transform of $Y_{2t}$. The Taylor coefficients of this inverse are afforded for instance by Lagrange inversion formula and we need to extract the $(n-1)$-th Taylor coefficients of 
 \begin{equation*}
u \mapsto \left[ \frac{(u+1)(u + 1-\alpha)}{(u + \alpha)}\right]^ne^{-(1+2u)n t}, \quad n \geq 1. 
 \end{equation*}
To proceed, we use Cauchy's Residue Theorem to expand:
\begin{equation*}
\left[\chi_t^{(\alpha)}(u)\right]^{-1}(u) = \sum_{n \geq 1} a_n^{(\alpha)}(t) (e^{-t}u)^n,  
\end{equation*}
where 
\begin{equation}\label{TCI}
a_n^{(\alpha)}(t) := \frac{1}{2i\pi n}\int_{\gamma} \left(1+\frac{1}{w}\right)^n \left[ \frac{(w + 1-\alpha)}{(w + \alpha)}\right]^n e^{-2ntw} dw
\end{equation}
and $\gamma$ is any simple curve inside the open disc of radius $\alpha$ and encircling the origin. Setting informally (we may take any determination of the logarithm):
\begin{equation*}
\phi_t^{(\alpha)}(w) := 2tw - \log\left[ \frac{(1+w)(w + 1-\alpha)}{w(w + \alpha)}\right],
\end{equation*}
then 
\begin{equation*}
a_n^{(\alpha)}(t) := \frac{1}{2i\pi n}\int_{\gamma} e^{-n\phi_t(w)} dw
\end{equation*}
so that the asymptotic behavior of $a_n^{(\alpha)}(t)$ as $n \rightarrow + \infty$ may be determined using the saddle point method. 

\subsection{Critical points of $\phi_t$}
The critical points of $\phi_t^{(\alpha)}$ are roots of the polynomial equation: 
\begin{equation*}
(1-\alpha)\left(w^2+w+\frac{\alpha}{2}\right) = -tw(1+w)(w+\alpha)(w+1-\alpha),
\end{equation*}
or equivalently 
\begin{equation*}
(w+\alpha)(w+1-\alpha)(tw(1+w) + 1-\alpha) = \alpha(1-\alpha)\left(\alpha- \frac{1}{2}\right).
\end{equation*}
In particular, if $\alpha =1/2$ then one retrieves the roots of $tw(1+w) + (1/2)$ (\cite{Gro-Mat}, p. 561). More generally, noting that $(w+\alpha)(w+1-\alpha) = w(1+w) + \alpha(1-\alpha)$, we can write this polynomial equation as 
\begin{equation}\label{PolyEquation}
t[Z(w)]^2 +(1-\alpha)(1+\alpha t)Z(w) + \frac{\alpha(1-\alpha)}{2} = 0, 
\end{equation}
where we set: 
\begin{equation*}
Z(w):= w(w+1) = \left(w+\frac{1}{2}\right)^2 - \frac{1}{4}.
\end{equation*}
 The discriminant of this $Z$-polynomial is computed as 
\begin{equation*}
\Delta_t^{(\alpha)} := (1-\alpha)[(1-\alpha)(1+\alpha t)^2 - 2\alpha t] 
\end{equation*}
and may take positive and negative values when $\alpha > 1/2$ (it is non negative otherwise). Actually, the roots of $\Delta_t^{(\alpha)}$ (as a polynomial in the time variable $t$) are given by $t_0 = t_0(\alpha) \leq 2 \leq t_1 = t_1(\alpha)$ displayed in \eqref{Times}. 
Consequently, the roots of the polynomial equation \eqref{PolyEquation} are negative real on $(0,t_0(\alpha)] \cup [t_1(\alpha), +\infty)$ and conjugate complex numbers with negative real part on $(t_0(\alpha), t_1(\alpha))$.
 
\begin{enumerate}
\item First case: $t \in [t_1(\alpha), +\infty)$ and the real roots are given by: 
\begin{equation*}
Z_{\pm}^{(\alpha)}(t) := \frac{-(1-\alpha)(1+\alpha t) \pm \sqrt{\Delta_t^{(\alpha)}}}{2t}.
\end{equation*}
Since $t \geq t_1(\alpha) \geq 2$ and since $4\alpha(1-\alpha) \leq 1$ then 
\begin{equation*}
\frac{t}{2} - (1-\alpha)(1+\alpha t) = \frac{t(1-2\alpha(1-\alpha)) - 2(1-\alpha)}{2} > 0 
\end{equation*}
holds true, whence we infer that $-1/4 \leq Z_{+}^{(\alpha)}(t)$. As to $Z_{-}^{(\alpha)}(t)$, we compute 
\begin{align*}
Z_{-}^{(\alpha)}(t) + \frac{1}{4} &= \frac{t-2(1-\alpha)(1+\alpha t) - 2\sqrt{\Delta_t^{(\alpha)}}}{4t}
\\& = \frac{[t-2(1-\alpha)(1+\alpha t)]^2 - 4\Delta_t^{(\alpha)}}{4t[t-2(1-\alpha)(1+\alpha t) + \sqrt{\Delta_t^{(\alpha)}}]}
\\& = \frac{t^2 - 4(1-\alpha)+ 4\alpha(1-\alpha)t}{4t[t-2(1-\alpha)(1+\alpha t) + \sqrt{\Delta_t^{(\alpha)}}]}
\\& =  \frac{t^2 + 4(1-\alpha)(\alpha t-1)}{4t[t-2(1-\alpha)(1+\alpha t) + \sqrt{\Delta_t^{(\alpha)}}]},
\end{align*}
which is positive since $\alpha t > 2\alpha \geq 1$. Accordingly, $Z_{-}^{(\alpha)}(t) > -1/4$ as well and in turn there are four real critical points of $\phi_t$ given by: 
\begin{equation*}
w_{\pm, \pm}^{(\alpha)}(t) = -\frac{1}{2} \pm \sqrt{\frac{1}{4} + Z_{\pm}^{(\alpha)}(t)}.
 \end{equation*}
Furthermore, it is obvious that $w_{+,\pm}^{(\alpha)} \in (-\alpha, 0)$ since $\alpha > 1/2$ and $Z_{\pm}^{(\alpha)}(t)$ are negative real numbers.  
However, the inequality 
\begin{equation*}
2t\alpha(1-\alpha) \geq  (1-\alpha)(1+\alpha t) + \sqrt{\Delta},
\end{equation*}
valid for any $\alpha \geq 1/2$ shows that the critical point 
\begin{equation*}
w_{-,-}^{(\alpha)}(t) =  -\frac{1}{2} - \sqrt{\frac{1}{4} + Z_{-}^{(\alpha)}(t)} \leq -\alpha,
\end{equation*}
and in turn $w_{-,+}^{(\alpha)}(t) <  w_{-,-}^{(\alpha)}(t) < -\alpha$ as well. In a nutshell, both critical points $w_{-,\pm}$  lie outside the open disc of radius $\alpha$ while $w_{+,\pm}$ lie inside it. 

\item Second case: $t \in (0,t_0(\alpha)]$ where we recall from \eqref{Times} the expression: 
\begin{equation*}
t_0(\alpha) = \frac{\alpha -\sqrt{2\alpha-1}}{\alpha(1-\alpha)} = \frac{1-\alpha}{\alpha(\alpha+\sqrt{2\alpha-1})}.
\end{equation*}
In this case, the inequality $-1/4 > Z_{+}^{(\alpha)}(t)$ holds true. Indeed, 
%However, If $\alpha \in ]1/2,1[$ then it is not difficult to see that 
\begin{equation*}
t - 2(1-\alpha)(1+\alpha t)  = t(1-2\alpha + 2\alpha^2) - 2(1-\alpha) \leq 0 
\end{equation*}
for any $t \in [0, t_0(\alpha))$ and 
%whence $Z_{-}^{(\alpha)}(t) \leq -1/4$. As a matter of fact, there are two real $w_{\pm,+}$ and two complex conjugate $w_{\pm,-}$ critical points lying on the imaginary line $\{\Re(z) = -1/2\}$. Moreover, noting that 
\begin{align*}
4\Delta - [t(1-2\alpha+2\alpha^2) - 2(1-\alpha)]^2  = -(2\alpha-1)^2 t + 4(1-\alpha)(1-2\alpha) < 0. 
\end{align*}
As a result, $Z_{-}^{(\alpha)}(t) \leq Z_{+}^{(\alpha)}(t) < -1/4$ so that $w_{\pm,+}$ and $w_{\pm,-}$ are two pairs of complex conjugate numbers whose real part equals $-1/2$. 
%and comparing $[(1-\alpha)(\alpha t-1)]^2$ and $\Delta$, we deduce that $w_{-,+} \leq -\alpha$ (of course, $w_{+,+} > -1/2 > -\alpha$). 
\item Third case: $t \in (t_0(\alpha), t_1(\alpha))$ where we also recall from \eqref{Times} the expression
\begin{equation*}
t_1(\alpha) = \frac{\alpha + \sqrt{2\alpha-1}}{\alpha(1-\alpha)} = \frac{1-\alpha}{\alpha(\alpha-\sqrt{2\alpha-1})}.
\end{equation*}
As mentioned above, $Z_{\pm}^{(\alpha)}(t)$ are complex conjugate and in turn there are four complex critical points $w_{\pm,\pm}^{(\alpha)}(t)$. Moreover, 
\begin{align*}
\frac{1}{4} - \frac{\alpha(1-\alpha)}{2} -\frac{1-\alpha}{2t_0(\alpha)} \leq \Re\left(\frac{1}{4} + Z_{\pm}^{(\alpha)}(t)\right)  \leq \frac{1}{4} - \frac{\alpha(1-\alpha)}{2} -\frac{1-\alpha}{2t_1(\alpha)}
%&= \frac{t(1-2\alpha + 2\alpha^2) - 2(1-\alpha)}{4t}  \leq \frac{(1-\alpha)\sqrt{2\alpha-1}}{2\alpha t(\alpha + \sqrt{2\alpha-1})}\left[-2\alpha - \sqrt{2\alpha-1}\right] < 0.
\end{align*}
which reads after few computations: 
\begin{align*}
\frac{1}{4} - \frac{\alpha}{2}[1+\sqrt{2\alpha-1}] \leq \Re\left(\frac{1}{4} + Z_{\pm}^{(\alpha)}(t)\right)  \leq \frac{1}{4} - \frac{\alpha}{2}[1-\sqrt{2\alpha-1}].
%&= \frac{t(1-2\alpha + 2\alpha^2) - 2(1-\alpha)}{4t}  \leq \frac{(1-\alpha)\sqrt{2\alpha-1}}{2\alpha t(\alpha + \sqrt{2\alpha-1})}\left[-2\alpha - \sqrt{2\alpha-1}\right] < 0.
\end{align*}
\end{enumerate}

\subsection{Asymptotic behavior of $a_n^{(\alpha)}(t)$ as $n \rightarrow +\infty$}
For sake of simplicity, we will only consider the case $t \in [t_1(\alpha), +\infty)$ since all the (four) real saddle points $w_{\pm,\pm}^{(\alpha)}(t)$ are real. In this respect, recall that  
$w_{-,\pm}^{(\alpha)}(t)$ lie outside the open disc of radius $\alpha$ while $w_{+,\pm}^{(\alpha)}(t)$ lie inside it. Using the saddle point method, we shall prove the following result which is the extension to $\alpha \geq 1/2$ of the asymptotic analysis performed in \cite{Gro-Mat}, p. 561: 
\begin{pro}
For any $\alpha \geq 1/2$ and $t \geq t_1(\alpha)$, one has 
\begin{equation*}
\left[\phi_t^{(\alpha)}\right]^{''}(w_{+,+}^{(\alpha)}(t)) < 0, \quad  \left[\phi_t^{(\alpha)}\right]^{''}(w_{+,-}^{(\alpha)}(t)) > 0,
\end{equation*}
%Asymptotically as $n \rightarrow +\infty$, %$a_n^{(\alpha)}(t) e^{-nt}$ is given by the contributions in a neighborhoods of $w_{+,+}$ and $w_{+,-}$:
and 
\begin{equation*}
a_n^{(\alpha)}(t) e^{-nt} \approx \frac{e^{-nt-n\Re[\phi_t^{(\alpha)}(w_{+,+}^{(\alpha)}(t))]}}{\sqrt{-2i\pi n \left[\phi_t^{(\alpha)}\right]^{''}(w_{+,+}^{(\alpha)}(t))}}
+\frac{e^{-nt-n\Re[\phi_t^{(\alpha)}(w_{+,-}^{(\alpha)}(t))]}}{\sqrt{2i\pi n |\left[\phi_t^{(\alpha)}\right]^{''}(w_{+,-}^{(\alpha)}(t))}}, \quad n \rightarrow +\infty.
\end{equation*}
Moreover, 
\begin{equation*}
t + \Re[\phi_t^{(\alpha)}(w_{+,\pm}^{(\alpha)}(t))] > 0.
\end{equation*}
In particular, the local inverse of $\chi_t^{(\alpha)}$ extends to an analytic function in the open unit disc. 
\end{pro}
%Since
%\begin{equation*}
%\left| \exp(\phi_t^{(\alpha)}(w_{+,-}^{(\alpha)}(t))) \right|>\left| \exp(\phi_t^{(\alpha)}(w_{+,+}^{(\alpha)}(t))) \right|, 
%\end{equation*}
%the saddle point $w_{+,-}^{(\alpha)}(t)$ is then more important important.
\begin{proof}
By the virtue of the Cauchy integral \eqref{TCI}, we shall discard the critical points $w_{-,\pm}^{(\alpha)}(t)$ of $\phi_t^{(\alpha)}$ and only focus on the contributions of $w_{+,\pm}^{(\alpha)}(t)$.
Near these points, the Taylor expansion of $\phi_t^{(\alpha)}$ takes the form: 
\begin{align*}
\phi_t^{(\alpha)}(z) & = \phi_t^{(\alpha)}(w_{+,\pm}^{(\alpha)}(t))+\frac{1}{2}{\phi_t^{(\alpha)}}''(w_{+,\pm}^{(\alpha)}(t))(z-w_{+,\pm}^{(\alpha)}(t))^2+\dots
\end{align*}
and we compute 
\begin{align}\label{SecDer}
\left[\phi_t^{(\alpha)}\right]''(w_{+,\pm}^{(\alpha)}(t)) & = \frac{1}{[1+w_{+,\pm}^{(\alpha)}(t)]^2} - \frac{1}{[w_{+,\pm}^{(\alpha)}(t)]^2} + \frac{1}{[1-\alpha +w_{+,\pm}^{(\alpha)}(t)]^2}-\frac{1}{[\alpha +w_{+,\pm}^{(\alpha)}(t)]^2} \nonumber 
\\& = -\frac{1+2w_{+,\pm}^{(\alpha)}(t)}{[1+w_{+,\pm}^{(\alpha)}(t)]^2[w_{+,\pm}^{(\alpha)}(t)]^2} + \frac{(2\alpha-1)(1+2w_{+,\pm}^{(\alpha))}(t))}{[1-\alpha +w_{+,\pm}^{(\alpha)}(t)]^2[\alpha +w_{+,\pm}^{(\alpha)}(t)]^2}.  
\end{align}
The rest of the proof is technical and long, and for ease of reading, we shall divide it into several lemmas. We start by proving that $\left[\phi_t^{(\alpha)}\right]''(w_{+,+}^{(\alpha)}(t)) < 0$. 
Since $1+2w_{+,+}^{(\alpha))}(t) \geq $ and since $2\alpha -1 \in [0,1]$, it only suffices to prove that:

 \begin{lem}
For any $\alpha \geq 1/2$ and any $t \geq t_1(\alpha)$, we have 
\begin{equation*}
1-\alpha +w_{+,+}^{(\alpha)}(t) \geq 0, 
\end{equation*}
and 
\begin{equation*}
[1-\alpha +w_{+,+}^{(\alpha)}(t)]^2[\alpha +w_{+,+}^{(\alpha)}(t)]^2 \geq (2\alpha-1)[1+w_{+,+}^{(\alpha)}(t)]^2[w_{+,+}^{(\alpha)}(t)]^2. 
\end{equation*}
\end{lem}
\begin{proof}
The first inequality is equivalent to 
\begin{equation*}
\frac{1}{2} + \sqrt{\frac{1}{4} + Z_{+}^{(\alpha)}(t)} > \alpha \quad \Leftrightarrow \quad w_{-,+}^{(\alpha)}(t) < -\alpha, 
\end{equation*}
which we already proved in the previous paragraph. Taking the square root of both sides in the second inequality and keeping in mind that $w_{+,+}^{(\alpha)}(t) \in [-1/2,0)$, we readily see that it is equivalent to 
\begin{equation*}
(1+\sqrt{2\alpha-1})\left[\left(w_{+,+}^{(\alpha)}(t) +\frac{1}{2}\right)^2 -\frac{1}{4}\right] +\alpha(1-\alpha)  \geq 0. 
\end{equation*}
Keeping in mind the expression of $Z_{+}^{(\alpha)}(t)$, the last inequality may be further rewritten as: 
\begin{equation*}
(1+\sqrt{2\alpha-1})Z_{+}^{(\alpha)}(t) + \alpha(1-\alpha) \geq 0.
\end{equation*}
Now, we shall prove that $t \mapsto -Z_{+}^{(\alpha)}(t)$ is decreasing. To this end, write $-2Z_{+}^{(\alpha)}(t)$ as 
\begin{equation*}
\alpha(1-\alpha) + \frac{1-\alpha}{t} - \frac{\sqrt{\Delta_t^{(\alpha)}}}{t} 
\end{equation*}
so that it suffices to prove that the map $t \mapsto \Delta_t^{(\alpha)}/t^2$ is increasing. In this respect, quick computations show that 
\begin{equation*}
\partial_t \frac{\Delta_t^{(\alpha)}}{t^2}  = 2(1-\alpha)\frac{\alpha^2 t -(1-\alpha)}{t^3} > 0 ,\quad  t\geq t_1(\alpha) > 2. 
\end{equation*}
As a result, we have the upper bound 
\begin{align*}
-Z_{+}^{(\alpha)}(t) \leq -Z_{+}^{(\alpha)}(t_1(\alpha)) & = \frac{(1-\alpha)(1+\alpha t_1(\alpha))}{2t_1(\alpha)} 
\\& = \alpha(1-\alpha) \frac{1+\sqrt{2\alpha-1}}{2\alpha + 2\sqrt{2\alpha-1}}
\\& = \frac{\alpha(1-\alpha)}{1+\sqrt{2\alpha-1}},
\end{align*}
which is equivalent to $(1+\sqrt{2\alpha-1})Z_{+}^{(\alpha)}(t) + \alpha(1-\alpha) \geq 0$. The lemma is proved.
\end{proof}

Regarding $w_{+,-}^{(\alpha)}(t)$, the counterpart of the previous lemma is as follows:
\begin{lem}
For any $\alpha \geq 1/2$ and any $t \geq t_1(\alpha)$, we have 
\begin{equation*}
1-\alpha +w_{+,-}^{(\alpha)}(t) \geq 0, 
\end{equation*}
and 
\begin{equation*}
[1-\alpha +w_{+,-}^{(\alpha)}(t)]^2[\alpha +w_{+,-}^{(\alpha)}(t)]^2 \leq (2\alpha-1)[1+w_{+,-}^{(\alpha)}(t)]^2[w_{+,-}^{(\alpha)}(t)]^2. 
\end{equation*}
In particular, $\left[\phi_t^{(\alpha)}\right]''(w_{+,-}^{(\alpha)}(t)) > 0$.
\end{lem}
\begin{proof}
We proceed along the same lines of the previous lemma. The first inequality is equivalent to 
\begin{equation*}
\frac{1}{2} + \sqrt{\frac{1}{4} + Z_{-}^{(\alpha)}(t)} > \alpha \quad \Leftrightarrow \quad w_{-,-}^{(\alpha)}(t) < -\alpha, 
\end{equation*}
which we also proved in the previous paragraph. Taking the square root of both sides in the second inequality and keeping in mind that $w_{+,-}^{(\alpha)}(t) \in [-1/2,0)$, we readily see that it is equivalent to 
\begin{equation*}
(1+\sqrt{2\alpha-1})\left[\left(w_{+,-}^{(\alpha)}(t) +\frac{1}{2}\right)^2 -\frac{1}{4}\right] +\alpha(1-\alpha)  = (1+\sqrt{2\alpha-1})Z_{-}^{(\alpha)}(t) + \alpha(1-\alpha) \leq 0. 
\end{equation*}
Now, we shall prove that $t \mapsto -Z_{-}^{(\alpha)}(t)$ is increasing. To proceed, we write
\begin{equation*}
-Z_{-}^{(\alpha)}(t) = \frac{\alpha(1-\alpha)}{-2tZ_{+}^{(\alpha)}(t)} = \frac{\alpha(1-\alpha)}{(1-\alpha)(1+\alpha t) - \sqrt{\Delta_t^{(\alpha)}}}
\end{equation*}
and it only remains to prove that the denominator is decreasing. In this respect, we compute its derivative with respect to the variable $t$: 
\begin{align*}
(1-\alpha)\left[\alpha - \frac{\alpha^2[(1-\alpha)t -1]}{\sqrt{\Delta_t^{(\alpha)}}}\right] = \alpha(1-\alpha)\left[ \frac{\sqrt{\Delta_t^{(\alpha)}}- \alpha[(1-\alpha)t -1]}{\sqrt{\Delta_t^{(\alpha)}}}\right].
\end{align*}
But it is easy to see that 
\begin{equation*}
\Delta_t^{(\alpha)}- \alpha^2[(1-\alpha)t -1]^2 = (1-\alpha)^2 - \alpha^2 \geq  0 
\end{equation*}
which implies that $\sqrt{\Delta_t^{(\alpha)}}\geq \alpha[(1-\alpha)t +1$ since $t \geq t_1(\alpha) > 1/(1-\alpha)$. Consequently, $t \mapsto -Z_{-}^{(\alpha)}(t)$ is increasing therefore:
\begin{align*}
-Z_{-}^{(\alpha)}(t) \geq -Z_{-}^{(\alpha)}(t_1(\alpha))  = \frac{\alpha(1-\alpha)}{1+\sqrt{2\alpha-1}}.
\end{align*}
Remembering \eqref{SecDer}, we infer that $\left[\phi_t^{(\alpha)}\right]''(w_{+,-}^{(\alpha)}(t)) > 0$ which finishes the proof. 
\end{proof}

Now, since $\left[\phi_t^{(\alpha)}\right]''(w_{+,+}^{(\alpha)}(t)) < 0$ (respectively, $\left[\phi_t^{(\alpha)}\right]''(w_{+,-}^{(\alpha)}(t)) > 0$), we must choose $\gamma$ such that the critical direction at $w_{+,+}^{(\alpha)}(t)$, given by $-\frac{\pi}{2} \pm \frac{\pi}{2}$, makes an angle less than $\frac{\pi}{4}$ with the tangent to the path $\gamma$ (respectively, the critical direction at $w_{+,-}^{(\alpha)}(t)$, given by $\pm \frac{\pi}{2}$, makes an angle less than $\frac{\pi}{4}$ with the tangent to the path $\gamma$).

For simplicity, we assume that $\gamma$ is chosen inside the open disc of radius $\alpha$ such that:
\begin{itemize}
\item The angle of the tangent to the path $\gamma$ at $w_{+,+}$ is less than $\frac{\pi}{4}$.
\item The angle of the tangent to the path $\gamma$ at $w_{+,-}$ makes an angle less than $\frac{\pi}{4}$ with the critical direction $-\frac{\pi}{2}$.
\end{itemize}

This ensures that, asymptotically as $n \rightarrow +\infty$, $a_n^{(\alpha)}(t) e^{-nt}$ is determined by contributions in the neighborhoods of $w_{+,+}$ and $w_{+,-}$:
\begin{equation*}
a_n^{(\alpha)}(t) e^{-nt} \approx \frac{e^{-nt-n\phi_t^{(\alpha)}(w_{+,+}^{(\alpha)}(t))}}{\sqrt{-2i\pi n \left[\phi_t^{(\alpha)}\right]''(w_{+,+}^{(\alpha)}(t))}}
+ \frac{e^{-nt-n\phi_t^{(\alpha)}(w_{+,-}^{(\alpha)}(t))}}{\sqrt{2i\pi n \left|\left[\phi_t^{(\alpha)}\right]''(w_{+,-}^{(\alpha)}(t))\right|}}, \quad n \to +\infty.
\end{equation*}
Since $w_{+,\pm}^{(\alpha)}(t) < 0$ then 
\begin{align*}
e^{-n\phi_t^{(\alpha)}(w_{+,\pm}^{(\alpha)}(t))} & = (-1)^n e^{-n\Re[\phi_t^{(\alpha)}(w_{+,\pm}^{(\alpha)}(t))]}
\end{align*}
where 
\begin{equation*}
\Re[\phi_t^{(\alpha)}(w_{+,\pm}^{(\alpha)}(t))] = \log\left\{\frac{[1+w_{+,\pm}^{(\alpha)}(t)][1-\alpha + w_{+,\pm}^{(\alpha)}(t)]}{[-w_{+,\pm}^{(\alpha)}(t)][\alpha + w_{+,\pm}^{(\alpha)}(t)]}\right\}
\end{equation*}
and the branch of the logarithm is now the real one. The remainder of the proof is then devoted to prove the two following inequalities:
\begin{equation*}
t + \Re[\phi_t^{(\alpha)}(w_{+,\pm}^{(\alpha)}(t))] > 0, \quad t \geq t_1(\alpha),
\end{equation*}
demonstrating the exponential decay of the Taylor coefficients of the (local) inverse of $\chi_t^{(\alpha)}$. We first prove that:
\begin{lem}
For any $\alpha > 1/2$ and any $t \geq t_1(\alpha)$, $t+ \Re[\phi_t^{(\alpha)}(w_{+,+}^{(\alpha)}(t))] > 0$.
 \end{lem}
\begin{proof}
Since $w_{+,+}^{(\alpha)}(t) < 0$, 
\begin{align*}
t + \Re[\phi_t^{(\alpha)}(w_{+,+}^{(\alpha)}(t))] & = t[1+2w_{+,+}^{(\alpha)}(t)] - \log\left\{\frac{[1+w_{+,+}^{(\alpha)}(t)][1-\alpha + w_{+,+}^{(\alpha)}(t)]}{[-w_{+,+}^{(\alpha)}(t)][\alpha + w_{+,+}^{(\alpha)}(t)]}\right\}
\\& = t\sqrt{1 + 4Z_{+}^{(\alpha)}(t)} - \log\left\{\frac{[1+\sqrt{1+4 Z_{+}^{(\alpha)}(t)}][\sqrt{1+4 Z_{+}^{(\alpha)}(t)}-(2\alpha-1)}{[1-\sqrt{1+4 Z_{+}^{(\alpha)}(t)}][\sqrt{1+4 Z_{+}^{(\alpha)}(t)}+(2\alpha-1)]}\right\}.
\end{align*}
Set 
\begin{equation*}
U_{+}^{(\alpha)}(t) := \sqrt{1+4 Z_{+}^{(\alpha)}(t)}
\end{equation*}
so that 
\begin{align*}
t + \Re[\phi_t^{(\alpha)}(w_{+,+}^{(\alpha)}(t))] = tU_{+}^{(\alpha)}(t)  - \log\left\{\frac{[1+ U_{+}^{(\alpha)}(t)][U_{+}^{(\alpha)}(t) -(2\alpha-1)]}{[1-U_{+}^{(\alpha)}(t) ][U_{+}^{(\alpha)}(t) +(2\alpha-1)]}\right\}.
\end{align*}
Noticing that 
\begin{align*}
\log\left\{\frac{[U_{+}^{(\alpha)}(t) -(2\alpha-1)]}{[U_{+}^{(\alpha)}(t) +(2\alpha-1)]}\right\} \geq 0,
\end{align*}
then we shall prove that the map 
\begin{align*}
t \mapsto tU_{+}^{(\alpha)}(t)  - \log\left\{\frac{1+ U_{+}^{(\alpha)}(t)}{1-U_{+}^{(\alpha)}(t)}\right\}, \quad t \geq t_1(\alpha),
\end{align*}
is increasing. To this end, recall that $t \mapsto Z_+^{(\alpha)}(t)$ (and in turn $t \mapsto U_+^{(\alpha)}(t)$) is increasing so that 
\begin{align*}
\partial_t\left[ tU_{+}^{(\alpha)}(t)  - \log\left\{\frac{[1+ U_{+}^{(\alpha)}(t)]}{[1-U_{+}^{(\alpha)}(t) ]}\right\} \right] & = U_{+}^{(\alpha)}(t) + \partial_t U_{+}^{(\alpha)}(t) \left[t - \frac{2}{1-[U_{+}^{(\alpha)}(t)]^2} \right] 
\\& = U_{+}^{(\alpha)}(t) +  \partial_t U_{+}^{(\alpha)}(t) \left[t + \frac{1}{2Z_{+}^{(\alpha)}(t)} \right] 
\\&  = U_{+}^{(\alpha)}(t) + t\left[\partial_t U_{+}^{(\alpha)}(t)\right] \frac{\sqrt{\Delta_t^{(\alpha)}}  - (1-\alpha)(1+\alpha t) + 1}{\sqrt{\Delta_t^{(\alpha)}}  - (1-\alpha)(1+\alpha t)} \geq 0. 
\end{align*}
Consequently, for any $t \geq t_1(\alpha)$, 
\begin{equation*}
tU_{+}^{(\alpha)}(t)  - \log\left\{\frac{[1+ U_{+}^{(\alpha)}(t)]}{[1-U_{+}^{(\alpha)}(t) ]}\right\} \geq t_1(\alpha)U_{+}^{(\alpha)}(t_1(\alpha))  - \log\left\{\frac{[1+ U_{+}^{(\alpha)}(t_1(\alpha))]}{[1-U_{+}^{(\alpha)}(t_1(\alpha))]}\right\}.
\end{equation*}
Now, recall also the expression: 
\begin{equation*}
Z_{+}^{(\alpha)}(t_1(\alpha)) = -\frac{\alpha(1-\alpha)}{1+\sqrt{2\alpha-1}},
\end{equation*}
whence we deduce that 
\begin{align*}
U_{+}^{(\alpha)}(t_1(\alpha)) = \sqrt{1+ 4Z_{+}^{(\alpha)}(t_1(\alpha))} = \sqrt{\frac{(1-2\alpha)^2 + \sqrt{2\alpha-1}}{1+\sqrt{2\alpha-1}}}.
\end{align*}
Besides, 
\begin{align*}
t_1(\alpha) = \frac{\alpha + \sqrt{2\alpha-1}}{\alpha(1-\alpha)} = 4\frac{[\alpha + \sqrt{2\alpha-1}]}{1-(1-2\alpha)^2} = 2 \frac{(1+\sqrt{2\alpha-1})^2}{1-(1-2\alpha)^2}. 
\end{align*}
Denoting $x := \sqrt{2\alpha-1} \in [0,1]$, we can write 
\begin{eqnarray*}
U_{+}^{(\alpha)}(t_1(\alpha)) & = & \sqrt{\frac{x^4 + x}{1+x}} = \sqrt{\frac{x^4 + x}{1+x}} = \sqrt{x-x^2+x^3}:= r(x), \\ 
t_1(\alpha) & = & 2 \frac{(1+x)^2}{1-x^4} = 2 \frac{(1+x)}{(1-x)(1+x^2)} = 2\frac{(1+x)}{1-(r(x))^2}.
\end{eqnarray*}
We are then led to prove that 
\begin{equation*}
2\frac{(1+x)r(x)}{1-(r(x))^2} - \ln\left[\frac{1+r(x)}{1-r(x)}\right] \geq 0, x \in [0,1].
\end{equation*}
This can be readily proved by proving that the map 
\begin{equation*}
y \mapsto \frac{2y}{1-y^2} - \ln\left[\frac{1+y}{1-y}\right], \quad y \in [0,1], 
\end{equation*}
is increasing and takes the zero value at $y=0$. The lemma is proved.
\end{proof}

In order to prove that $t+ \Re[\phi_t^{(\alpha)}(w_{+,-}^{(\alpha)}(t))] > 0$, we need the following lemma:
\begin{lem}\label{variations of Z}
For any $t \geq 2$, let $\alpha(t)$ be the unique real number in $[1/2,1]$ such that $t_1(\alpha(t)) = t$. Then $\alpha\mapsto Z_{-}^{(\alpha)}(t)$ is increasing on $[1/2,\alpha(t)]$. 
\end{lem}
\begin{proof}
The uniqueness of $\alpha(t)$ follows from the fact that $\alpha \mapsto t_1(\alpha)$ is increasing from $[1/2,1]$ onto $[2, +\infty)$. Now, recall the expression:
\begin{equation*}
Z_{-}^{(\alpha)}(t) = \frac{(1-\alpha)(1+\alpha t) + \sqrt{\Delta_t^{(\alpha)}}}{-2t}. 
\end{equation*}
so that it only remains to prove that the numerator is decreasing on $[1/2,\alpha(t)]$. In this respect, we compute its derivative with respect to the variable $\alpha$: 
\begin{align*}
-1-t(2\alpha-1)+\frac{-2(1+t(2\alpha-1))(1-\alpha)(1+\alpha t)-2t(1-2\alpha)}{2\sqrt{\Delta_t^{(\alpha)}}},
\end{align*}
which may be written as
\begin{align*}
-\frac{\left[\sqrt{\Delta_t^{(\alpha)}} + (1-\alpha)(1+\alpha t)\right]+t(2\alpha-1)\left[\sqrt{\Delta_t^{(\alpha)}} +(1-\alpha)(1+\alpha t)-1\right]}{\sqrt{\Delta_t^{(\alpha)}}}.
\end{align*}
Since $t = t_1(\alpha(t))$ and $\alpha \in [1/2, \alpha(t)]$ then $t \geq t_1(\alpha) \geq 1/(1-\alpha)$ which implies that 
\begin{equation*}
(1-\alpha)(1+\alpha t)-1 = \alpha [(1-\alpha)t - 1] \geq 0
\end{equation*}
and proves the lemma. 
%This is negative since $(1-\alpha)(1+\alpha t) -1 = \alpha[t(1-\alpha) -1] > 0$ and $t_1(\alpha) > 1/(1-\alpha)$. 
\end{proof}

%\begin{rem}
%The same computations as above show that $\alpha\mapsto Z_{+}^{(\alpha)}(t)$ is decreasing. More precisely,
%\begin{align*}
%\partial_{\alpha}[Z_{+}^{(\alpha)}(t)] & = \frac{1}{2t} \partial_{\alpha}\left\{\sqrt{\Delta_t^{(\alpha)}}-(1-\alpha)(1+\alpha t)\right\} 
%\\& = 1+t(2\alpha-1) - \frac{(1-\alpha)(1+\alpha t)+t(2\alpha-1)((1-\alpha)(1+\alpha t)-1)}{\sqrt{\Delta_t^{(\alpha)}}}
%\\& =  \frac{\sqrt{\Delta_t^{(\alpha)}} - (1-\alpha)(1+\alpha t)+t(2\alpha-1)[\sqrt{\Delta_t^{(\alpha)}} - (1-\alpha)(1+\alpha t)+1)}{\sqrt{\Delta_t^{(\alpha)}}}.
%\end{align*}
%But it is easy to prove that $\sqrt{\Delta_t^{(\alpha)}} - (1-\alpha)(1+\alpha t)+1 \leq 0$ whence we deduce that $\partial_{\alpha}[Z_{+}^{(\alpha)}(t)] \leq 0$.
%\end{rem}

\begin{cor}
For any $\alpha \geq 1/2$ and any $t \geq t_1(\alpha)$, $t+ \Re[\phi_t^{(\alpha)}(w_{+,-}^{(\alpha)}(t))] > 0$.
\end{cor}
\begin{proof}
The proof is similar to the one written for proving $t+ \phi_t^{(\alpha)}(w_{+,+}^{(\alpha)}(t)) > 0$. Indeed, 
\begin{align*}
t + \Re[\phi_t^{(\alpha)}(w_{+,-}^{(\alpha)}(t))] = tU_{-}^{(\alpha)}(t)  - \log\left\{\frac{[1+ U_{-}^{(\alpha)}(t)][U_{-}^{(\alpha)}(t) -(2\alpha-1)]}{[1-U_{-}^{(\alpha)}(t) ][U_{-}^{(\alpha)}(t) +(2\alpha-1)]}\right\},
\end{align*}
and it suffices to prove that 
\begin{equation*}
 tU_{-}^{(\alpha)}(t)  - \log\left\{\frac{[1+ U_{-}^{(\alpha)}(t)]}{[1-U_{-}^{(\alpha)}(t) ]}\right\} > 0, \quad t \geq t_1(\alpha). 
 \end{equation*}
To this end, notice that $t \geq t_1(\alpha) \Leftrightarrow \alpha \leq \alpha(t)$. Besides, we compute 
\begin{align*}
\partial_\alpha\left[ tU_{-}^{(\alpha)}(t)  - \log\left\{\frac{[1+ U_{-}^{(\alpha)}(t)]}{[1-U_{-}^{(\alpha)}(t) ]}\right\} \right] & =  \partial_\alpha U_{-}^{(\alpha)}(t) \left[t - \frac{2}{1-[U_{-}^{(\alpha)}(t)]^2} \right] 
\\& =  \partial_\alpha U_{-}^{(\alpha)}(t) \left[t + \frac{1}{2Z_{-}^{(\alpha)}(t)} \right] 
\\&  =  t\partial_\alpha U_{-}^{(\alpha)}(t) \frac{\sqrt{\Delta_t^{(\alpha)}}  + (1-\alpha)(1+\alpha t) - 1}{\sqrt{\Delta_t^{(\alpha)}}  + (1-\alpha)(1+\alpha t)}. 
\end{align*}
But the ratio 
\begin{equation*}
\frac{\sqrt{\Delta_t^{(\alpha)}}  + (1-\alpha)(1+\alpha t) - 1}{\sqrt{\Delta_t^{(\alpha)}}  + (1-\alpha)(1+\alpha t)},
\end{equation*}
is non negative, hence Lemma \eqref{variations of Z} implies that 
\begin{equation*}
\alpha \mapsto tU_{-}^{(\alpha)}(t)  - \log\left\{\frac{[1+ U_{-}^{(\alpha)}(t)]}{[1-U_{-}^{(\alpha)}(t) ]}\right\}  
\end{equation*}
is increasing on $[1/2, \alpha(t)]$. It follows that:
\begin{align*}
tU_{-}^{(\alpha)}(t)  - \log\left\{\frac{[1+ U_{-}^{(\alpha)}(t)]}{[1-U_{-}^{(\alpha)}(t) ]}\right\} \geq tU_{-}^{(1/2)}(t)  - \log\left\{\frac{[1+ U_{-}^{(1/2)}(t)]}{[1-U_{-}^{(1/2)}(t) ]}\right\} =0, 
\end{align*}
as required. 
\end{proof}
The proposition is now proved. 
\end{proof}

\section{Kunisky's identity: proof and dynamical extension} 
\subsection{Reminder}
The free MANOVA distribution, discovered by K. W. Wachter in the late seventies (\cite{Wac}), is the free multiplicative convolution of two free Bernoulli distributions of parameters $(\beta, \alpha)$ and has the following Lebesgue decomposition (\cite{Voi}, \cite{Aub}): 
\begin{equation*}
\nu^{(\beta,\alpha)}(dx) := \left(1-\min\left(\beta,\alpha\right)\right)\delta_0 + \max\left(0,\alpha + \beta -1\right)\delta_1 + \frac{\sqrt{(x_+-x)(x-x_-)}}{2x(1-x)}{\bf 1}_{[x_-,x_+]}(x)dx,
\end{equation*}
where
\begin{equation*}
x_{\pm} = x_{\pm}(\beta, \alpha) := \left(\sqrt{\alpha(1-\beta)} \pm \sqrt{\beta(1-\alpha)}\right)^2.
\end{equation*}

Equivalently, $\nu^{(\beta,\alpha)}$ is the spectral distribution of the angle operator of two free orthogonal projections which may be realized as $PUQU^{\star}P$, where $U$ is a Haar unitary operator which is free with $\{P,Q\}$ in $\mathscr{A}$. 
Considered as an operator in the compressed algebra $P\mathscr{A}P$, $PUQU^{\star}P$ is the weak (and strong) limit as $t \rightarrow +\infty$ of $J_t$ and as such, its spectral distribution is the weak limit of $\mu_t^{(\beta,\alpha)}$ as $t \rightarrow +\infty$
(see e.g. \cite{Dem}):  
\begin{equation*}
\mu_{\infty}^{(\beta,\alpha)}(dx) = \max\left(0, 1- \frac{\alpha}{\beta}\right)\delta_0 + \max\left(0,\frac{\alpha + \beta -1}{\beta}\right)\delta_1 + \frac{\sqrt{(x_+-x)(x-x_-)}}{2\beta x(1-x)}{\bf 1}_{[x_-,x_+]}(x)dx, 
\end{equation*}

In \cite{Kun}, Proposition C.2, it was observed that the density of $\nu^{(\alpha,\alpha)}$ normalized to have unit mass is the pushforward of the normalized density of $\nu^{(1/2, \alpha)}$ under the map $x \mapsto (2x-1)^2$. 
This observation was obtained in \cite{Kun} using asymptotic freeness and can be checked directly from the explicit expression of the density of $\nu^{(\beta, \alpha)}$ displayed above. 
Nonetheless, the author of \cite{Kun} raised the question whether an explanation of it relying on arguments from free probability theory may be provided. In the next paragraph, we provide such an explanation relying on the Nica-Speicher semi-group for the compression by a free orthogonal projection with rank $1/2$. We shall also need an identity we proved in \cite{Dem-Ham} and valid for any orthogonal projections without any freeness assumption.
As before, we assume without loss of generality that $\alpha \geq 1/2$ since otherwise we may consider $P-PUQU^{\star}P = PU({\bf 1}-Q)U^{\star}P$. 

\subsection{Another proof of Kunisky's identity}
Straightforward computations show that the densities of $\nu^{(1/2, \alpha)}$ and of $\mu_{\infty}^{(1/2, \alpha)}$ normalized to have unit mass coincide, and likewise for the normalized densities of 
$\nu^{(\alpha, \alpha)}$ and of $\mu_{\infty}^{(\alpha, \alpha)}$. Consequently, Proposition C.2. in \cite{Kun} may be reformulated as: 
\begin{pro}\label{KunBis}
The normalized density of $\mu_{\infty}^{(\alpha, \alpha)}$ is the pushforward of the normalized density of $\mu_{\infty}^{(1/2, \alpha)}$ under the map $x \mapsto (2x-1)^2$. 
\end{pro}
%rather than $\mathscr{A}$ which is more suited to our purposes. 
\begin{proof}
Let $f_{\infty}^{(\beta, \alpha)}$ denote the (non normalized) density of $\mu_{\infty}^{(\beta, \alpha)}$. Then its mass equals $2(1-\alpha)$  (the weight of 
$\mu_{P\mathscr{A}P}^{(1/2, \alpha)}$ at $x=0$ vanishes since we assumed that $\alpha \geq 1/2$) so that:
% and likewise for general ranks $(\alpha, \beta)$). In this respect, we have for any $j \geq 0$, 
\begin{align*}
\frac{1}{2(1-\alpha)} \int(2x-1)^{2j}f_{\infty}^{(1/2, \alpha)}(x) dx & = \frac{1}{2(1-\alpha)}\int(2x-1)^{2j} \mu_{\infty}^{(1/2, \alpha)}(dx) - \frac{2\alpha-1}{2(1-\alpha)}. 
\end{align*}
Now, recall from \cite{Nic-Spe} that the distribution of $PUQU^{\star}P$ in $(P\mathscr{A}P, 2\tau)$ coincides with that of $(Q_1+Q_2)/2$ in $(\mathscr{A}, \tau)$, where $Q_i, i \in \{1,2\},$ are two free copies of $Q$. Consequently, 
\begin{align*}
\frac{1}{2(1-\alpha)} \int(2x-1)^{2j} f_{\infty}^{(1/2, \alpha)}(x)dx & = \frac{1}{2(1-\alpha)}\tau[(Q_1+Q_2-1)^{2j}]- \frac{2\alpha-1}{2(1-\alpha)}. 
\end{align*}
At this level, we will not need any additional argument from free probability theory. Rather, we shall appeal to Theorem 1 in \cite{Dem-Ham} which is valid for any pair of orthogonal projections. An application of this theorem leads to the following identity:
\begin{equation*}
\tau[(Q_1+Q_2-1)^{2j} = 2\tau[(Q_1Q_2Q_1)^j] - (2\alpha-1).
\end{equation*}
As a result,
\begin{align*}
\frac{1}{2(1-\alpha)} \int(2x-1)^{2j}f_{\infty}^{(1/2, \alpha)}(x)dx & = \frac{1}{(1-\alpha)}\tau[(Q_1Q_2Q_1)^{j}]- \frac{2\alpha-1}{1-\alpha} 
\\& = \frac{\alpha}{(1-\alpha)}\left[\frac{\tau[(Q_1Q_2Q_1)^{j}}{\alpha} - \frac{2\alpha-1}{\alpha}  \right]
\\& = \frac{\alpha}{(1-\alpha)}\left[\int x^j \left(\mu_{\infty}^{(\alpha, \alpha)}- \frac{2\alpha-1}{\alpha}\delta_1\right)(dx) \right] 
\\& =  \frac{\alpha}{(1-\alpha)}\left[\int x^j f_{\infty}^{(\alpha, \alpha)}(x)dx\right],
\end{align*}
proving Proposition C.2. in \cite{Kun}.
\end{proof}
In the same paper, the author also asked whether other distributional identities among free MANOVA laws exist and we do not have any answer to this question for the time being. 
Indeed, the Nica-Speicher semi-group remains valid for $\tau(P) = 1/k, k \geq 2,$ while we do not know how Theorem 1 in \cite{Dem-Ham} extends to $k \geq 3$ free copies of $Q$. 
We would like also to point out that since the commutative prototype  of $\mu_{\infty}^{(\beta,\alpha)}$ is the beta distribution, Proposition C.2. in \cite{Kun} reminds the quadratic transformation relating Jacobi polynomials of parameters 
$(1/2, \lambda)$ to ultraspherical polynomials of parameter $\lambda)$ (see e.g. \cite{Koo}, eq. (2.3)). To the best of our knowledge, no higher order extensions of this transformation exist in literature.

\subsection{Extension to the free Jacobi process}
We now proceed to the dynamical extension of Proposition \ref{KunBis}. To this end, denote $G_t^{(1/2,\alpha)}$ the Cauchy-Stieltjes transform of $\mu_t^{(1/2,\alpha)}$ and recall the Cauchy-Stieltjes transform 
$G_t^{(\alpha)}$ of $\mu_t^{(\alpha,\alpha)}$. Denote further $k_{t}^{(1/2,\alpha)}$ and $k_{t}^{(\alpha, \alpha)}$ the normalized densities of these probability  distributions. In this respect, a major step toward the proof of Proposition \ref{Pr1} is the following result:

%Note that since $J_t^{(\beta,\alpha)}$ is a non negative bounded (by one) operator, then the Cauchy-Stieltjes transform of its spectral distribution is analytic in $\mathbb{C} \setminus [0,1]$. However, the density of this probability measure is far from being explicit except when both projections have equal rank $1/2$. Nonetheless, its Cauchy-Stieltjes transform satisfies a transport-type partial differential equation (\cite{Dem}) which we shall use to prove the following proposition:
\begin{pro}\label{Pro1}
Let $h_{t}^{(1/2,\alpha)}$ be the pushforward of $k_{t}^{(1/2,\alpha)}$ under the map $x \mapsto 2x-1$ and assume it is an even function. Then the Cauchy-Stieltjes transforms of the pushforward of $h_{t/2}^{(1/2,\alpha)}$ under the map $x \mapsto x^2$ and 
of $k_{t}^{(\alpha, \alpha)}$ satisfy the same pde. 
 \end{pro}
%In this respect, note that we can equivalently state that the compression of $Q$ (or any other self-adjoint operator) by a free projection $P$ as giving the same distribution of $Q_1+UQ_2U^{\star}$, where $Q_1$ and $Q_2$ are not necessarily free copies of $Q$ and $U$ is a Har unitary operator in $\mathscr{A}$. 

%The spectral distribution, say $\mu_t^{(\alpha, \beta)}$, of $PY_tQY_t^{\star}P$ in $(P\mathscr{A}P, \tau/\tau(P))$ is not as explicit as $\mu_{P\mathscr{A}P}(dx)$. Nonetheless, both probability measures share the same atoms and weights, and admit further only absolutely continuous parts. Rather than using the freeness assumption, the dynamical setting allows to appeal to the partial differential equation satisfied by the Cauchy-Stieltjes of $\mu_t^{(\alpha, \beta)}$: 
%\begin{equation*}
%G_t^{(\alpha,\beta)}(z) := \int_0^1\frac{1}{z-x}\mu_t^{(\alpha, \beta)}(dx). 
%\end{equation*}
\begin{proof}
%For any $z \in \mathbb{C} \setminus [0,1]$, the Cauchy-Stieltjes transform of $\mu_{t, P\mathscr{A}P}^{(1/2,\alpha)}$ is defined by
%\begin{equation*}
%G_t^{(1/2,\alpha)}(z) := \int \frac{1}{z-x}\mu_{t, P\mathscr{A}P}^{(1/2,\alpha)}(dx).
%\end{equation*}
Recall from \cite{Dem}, Proposition 6, that $tG_t^{(1/2, \alpha)}$ satisfies the following pde: 
\begin{equation*}
\partial_tG_t^{(1/2, \alpha)}(z)= \frac{1}{2}\partial_z\left[(1-2\alpha)G_t^{(1/2, \alpha)}(z) +  z(z-1)[G_t^{(1/2, \alpha)}]^2(z) \right], \quad z \in \mathbb{C} \setminus [0,1].
\end{equation*}
Since $\alpha \geq 1/2$ then Theorem 1 in \cite{Ham0} shows (after normalization) that $\mu_{t}^{(1/2,\alpha)}$ has a zero weight at $x=0$ while it assigns the weight $2\alpha-1$ to $x=1$. 
Accordingly, we shall consider
\begin{align*}
\tilde{G}_t^{(1/2,\alpha)}(z) :=  \int_0^1\frac{1}{z-x} k_{t}^{(1/2,\alpha)}(x) dx
& = \frac{1}{2(1-\alpha)} \int_0^1\frac{1}{z-x}\left[\mu_{t}^{(1/2, \alpha)}(dx) - (2\alpha-1)\delta_1\right] 
\\& = \frac{1}{2(1-\alpha)} \left[G_t^{(1/2, \alpha)}(z) - \frac{(2\alpha-1)}{z-1}\right].
\end{align*}
Quick computations then shows that 
\begin{equation*}
\partial_t\tilde{G}_t^{(1/2, \alpha)}(z)= \partial_z\left[(2\alpha-1)\left(z-\frac{1}{2}\right)\tilde{G}_t^{(1/2, \alpha)}(z) +  (1-\alpha) z(z-1)[\tilde{G}_t^{(1/2, \alpha)}]^2(z) \right].
\end{equation*}
In turn, the Cauchy-Stieltjes transform of $h_{t}^{(1/2, \alpha)}$ is expressed through $\tilde{G}_t^{(1/2,\alpha)}$ as: 
\begin{equation*}
u_t^{(1/2,\alpha)}(z) :=  \int_0^1\frac{1}{z-x}h_{t}^{(1/2, \alpha)}(x) dx = \frac{1}{2}\tilde{G}_t^{(1/2,\alpha)}\left(\frac{z+1}{2}\right),
\end{equation*}  
so that 
\begin{equation*}
\partial_tu_t^{(1/2, \alpha)}(z)= \partial_z\left[(2\alpha-1)zu_t^{(1/2, \alpha)}(z) +  (1-\alpha) (z^2-1)[u_t^{(1/2, \alpha)}]^2(z) \right].
\end{equation*}
But since $h_{t}^{(1/2, \alpha)}$ is assumed to be an even function, then we may write 
\begin{equation*}
u_t^{(1/2, \alpha)}(z) =: zv_t^{(1/2, \alpha)}(z^2)
\end{equation*}
 where now $v_t^{(1/2, \alpha)}$ is the Cauchy-Stieltjes transform of the pushforward of $h_t^{(1/2, \alpha)}$ under the map $x \mapsto x^2$. To see this, it suffices to write 
 \begin{align*}
 u_t^{(1/2,\alpha)}(z) = \int \frac{1}{z-x}h_{t}^{(1/2, \alpha)}(x)dx &= \frac{1}{2}\int\left[\frac{1}{z-x} + \frac{1}{z+x}\right]h_{t}^{(1/2, \alpha)}(x) dx 
 \\&= z\int \frac{1}{z^2-x^2}h_{t}^{(1/2, \alpha)}(x)dx
 \\& = z \int\frac{1}{z^2-(2x-1)^2}k_{t}^{(1/2, \alpha)}(x)dx.
 \end{align*}
 
 As a matter of fact 
\begin{align*}
z \partial_tv_t^{(1/2, \alpha)}(z^2) &= \partial_z\left[(2\alpha-1)z^2v_t^{(1/2, \alpha)}(z^2) +  (1-\alpha) z^2(z^2-1)[v_t^{(1/2, \alpha)}]^2(z) \right] 
\\& = 2z\partial_z\left[(2\alpha-1)zv_t^{(1/2, \alpha)} +  (1-\alpha) z(z-1)[v_t^{(1/2, \alpha)}]^2 \right](z^2),
\end{align*}
or equivalently 
\begin{align*}
\partial_tv_t^{(1/2, \alpha)}(z) = 2\partial_z\left[(2\alpha-1)zv_t^{(1/2, \alpha)}(z) +  (1-\alpha) z(z-1)[v_t^{(1/2, \alpha)}(z)]^2 \right].
\end{align*}
Performing the time-change $t \mapsto t/2$, we get 
\begin{align}\label{PDE1}
\partial_tv_{t/2}^{(1/2, \alpha)}(z) = \partial_z\left[(2\alpha-1)zv_{t/2}^{(1/2, \alpha)}(z) +  (1-\alpha) z(z-1)[v_{t/2}^{(1/2, \alpha)}(z)]^2 \right].
\end{align}
Finally, the Cauchy-Stieltjes transform of $\mu_{t, P\mathscr{A}P}^{(\alpha, \alpha)}$ satisfies: 
\begin{equation*}
\partial_tG_t^{(\alpha, \alpha)}(z)= \partial_z\left[(1-2\alpha)zG_t^{(\alpha, \alpha)}(z) + \alpha z(z-1)[G_t^{(\alpha, \alpha)}]^2(z) \right], \quad z \in \mathbb{C} \setminus [0,1].
\end{equation*}
Removing the weight assigned by $\mu_{t, P\mathscr{A}P}^{(\alpha, \alpha)}$ to $x=1$ (the weight at $x=0$ vanishes), we are led to consider the Cauchy-Stieltjes distribution of its normalized density:
\begin{align*}
\tilde{G}_t^{(\alpha,\alpha)}(z) := \int_0^1\frac{1}{z-x} k_{t}^{(\alpha,\alpha)}(x) dx & = \frac{\alpha}{(1-\alpha)} \int_0^1\frac{1}{z-x}\left[\mu_{t}^{(\alpha, \alpha)}(dx) - \frac{(2\alpha-1)}{\alpha}\delta_1\right] 
\\&= \frac{\alpha}{1-\alpha} \left[G_t^{(\alpha, \alpha)}(z) - \frac{(2\alpha-1)}{\alpha(z-1)}\right].
\end{align*}
Straightforward computations then show that 
\begin{equation}\label{PDE2}
\partial_t\tilde{G}_t^{(\alpha,\alpha)}(z) = \partial_z[(2\alpha-1)z\tilde{G}_t^{(\alpha, \alpha)}(z) + (1-\alpha) z(z-1)[\tilde{G}_t^{(\alpha, \alpha)}]^2(z)],
\end{equation}
which is the same as \eqref{PDE1}. 
\end{proof}

The following lemma shows that the assumption made in Proposition \ref{Pr1} on the angle operators ensure that $h_{t}^{(1/2,\alpha)}$ is an even function.  
\begin{lem}
If $PQP$ and $({\bf 1}-P)Q({\bf 1}-P)$ have the same spectral distribution in the compressed probability space $(P\mathscr{A}P, 2\tau)$ then $h_{t}^{(1/2,\alpha)}$ is an even function. 
\end{lem}

\begin{proof}
Let 
\begin{equation*}
R := 2P-{\bf 1}, \quad S := 2Q-{\bf 1},
\end{equation*}
 be the orthogonal symmetries associated to $P$ and $Q$ respectively. Then both operators are unitary and $\tau(R) = 0$. From \cite{Dem-Ham}, we readily see that the moments of $PQP$ 
are encoded only by those of $RS$ and by $\tau(S)$: for any $j \geq 0$,
\begin{equation*}
\tau[(PQP)^j] = \frac{1}{2^{2j+1}}\binom{2j}{j} + \frac{\tau(S)}{4} + \frac{1}{2^{2j}} \sum_{k=1}^j\binom{2j}{j-k} \tau[(RS)^k]. 
\end{equation*}
Since ${\bf 1}-P = ({\bf 1}-R)/2$ then we deduce that $PQP$ and $({\bf 1}-P)Q({\bf 1}-P)$ have the spectral distribution if and only if the unitary operators $RS$ and $-RS$ do so (i.e. $RS$ is an even operator). 

%and assume that have the same spectral distributions (or equivalently that the odd moments of $RS$ vanish)
Now, recall from Theorem 1.1 in \cite{Ham0} that the density of $\mu_{t}^{(1/2, \alpha)}(dx)$ is given by: 
\begin{equation*}
\frac{\kappa_t(e^{i2\arccos(\sqrt{x})})}{2\pi \sqrt{x(1-x)}}, \quad x \in (0,1),
\end{equation*}
where $\kappa_t$ is the density of the spectral distribution of the unitary operator $RY_tSY_t^{\star}$. Performing the variable change $u = 2x-1$ transforms this density into 
\begin{equation*}
\frac{\kappa_t(e^{i2\arccos(\sqrt{(1+u)/2})})}{\pi \sqrt{(1-u^2)}}, \quad x \in (0,1),
\end{equation*}
and it remains to show that the latter is an even function in the variable $u$. To this end, we use the relation 
\begin{equation*}
2\arccos(\sqrt{(1-u)/2}) = \pi - 2\arccos(\sqrt{(1+u)/2}),
\end{equation*}
and need to prove that $\kappa_t(-w) = \kappa(w), |w| = 1$. But the last property follows from Proposition 2.1 in \cite{Ham1} which shows that the odd moments of $RY_tSY_t^{\star}$ vanish since those of $RS$ do by assumption..
%(it also from the fact that $(R,S)$ are $\star$-free from $(Y_t)_{t \geq 0}$ since . It also follows. 
%Note that $\kappa_t$ is symmetric under the flip $e^{i\theta} = e^{-i\theta}$ since $RY_tSY_t^{\star}$ has the same distribution as $RY_t^{\star}SY_t$.
\end{proof}
%Since $\tau(P = 1/2$ then $\tau(R)=0$. Consequently, (\cite{Dem-Ham}). Likewise, the moments of $({\bf 1} - P)Q({\bf 1}-P)$ are encoded by those of $-RS$ and by $\tau(S)$. 
%All that to say that the assumption of the lemma may be rephrased as $PQP$ and $({\bf 1} - P)Q({\bf 1}-P)$ having the same spectral distribution in $(P\mathscr{A}P, 2\tau)$ which is the assumption of our proposition. 
%Thus we can apply the lemma above and deduce that 

With the help of this lemma, the statement of Proposition \ref{Pro1} holds true whence Proposition \ref{Pr1} follows. 
 
\begin{rem}
The assumption that the unitary operators $RS$ and $-RS$ share the same spectral distribution is readily satisfied when $R$ and $S$ are free. We do not know how to construct a concrete example going beyond freeness. 
Nonetheless, we can exhibit probability distributions on the unit circle which are invariant under complex conjugation $w \mapsto \overline{w}$ and under the sign flip $w \mapsto -w$. 
\end{rem}

\section{Matching Initial data: going beyond freeness}
Even though the pdes \eqref{PDE1} and \eqref{PDE2} are the same, we can not infer that their solutions coincide unless their initial data at $t=0$ agree. At the level of moments, the agreement of initial data is equivalent to
\begin{equation}\label{Equa3}
2\tau[(2PQP-P)^{2j}] = \tau((Q_1+Q_2 - {\bf 1})^{2j}),  \quad j \geq 0,
\end{equation}
where $Q_1$ and $Q_2$ are orthogonal projections in $\mathscr{A}$ with common rank $\alpha = \tau(Q)$. To see this, recall from the proof of Proposition \ref{Pro1} that for any $t > 0$, $v_{t/2}^{(1/2,\alpha)}$ is the Cauchy-Stieltjes transform of the pushforward of the density $f_t^{(1/2, \alpha)}$ 
under the map $x \mapsto (2x-1)^2$. Since 
\begin{equation*}
f_t^{(1/2, \alpha)}(x) dx = \frac{1}{2(1-\alpha)}\left[\mu_{t,P\mathscr{A}P}^{(1/2, \alpha)}(dx)-(2\alpha-1)\delta_1\right], 
\end{equation*}
then $v_{t/2}^{(1/2,\alpha)}$ may be expanded as 
\begin{equation*}
v_{t/2}^{(1/2,\alpha)}(z) = \frac{1}{2(1-\alpha)}\sum_{j = 0}^{\infty}\frac{1}{z^{j+1}} \left[2\tau[(2PY_{t/2}QY_{t/2}^{\star}P-P)^{2j}] - (2\alpha-1)\right], \quad t \geq 0.
\end{equation*}
Similarly, for any $t \geq 0$, the expansion 
\begin{align*}
\tilde{G}_t^{(\alpha,\alpha)}(z) & = \frac{\alpha}{(1-\alpha)} \int_0^1\frac{1}{z-x}\left[\mu_t^{(\alpha, \alpha)}(dx) - \frac{(2\alpha-1)}{\alpha}\delta_1\right] 
\\& = \frac{\alpha}{1-\alpha}\sum_{j=0}^{\infty}\frac{1}{z^{j+1}}\left[\frac{\tau[(Q_1Y_tQ_2Y_t^{\star}Q_1)^{j}]}{\alpha} - \frac{(2\alpha-1)}{\alpha}\right],
\end{align*}
holds. Consequently, the equality $v_{0}^{(1/2,\alpha)}(z) = \tilde{G}_0^{(\alpha,\alpha)}(z)$ is equivalent to 
\begin{equation}\label{Equa1}
\tau[(2PQP-P)^{2j}] - \frac{(2\alpha-1)}{2} = \tau[(Q_1Q_2Q_1)^{j}] - (2\alpha-1), 
\end{equation}
for any $j \geq 0$. Besides, Theorem 1 in \cite{Dem-Ham} yields again
\begin{equation*}
 \tau[(Q_1Q_2Q_1)^{j}] - \frac{(2\alpha-1)}{2} = \frac{1}{2}\tau(Q_1+Q_2 - {\bf 1})^{2j}),
 \end{equation*}
 so that \eqref{Equa1} is equivalent to \eqref{Equa3}. 
 % Finally, note also that (\cite{Dem-Ham}, section 5)
%\begin{equation*}
%\tau[(Q_1+Q_2 - {\bf 1})^{2j+1}] = \tau(Q_1+ Q_2Y_t^{\star} - {\bf 1}) = 2\alpha-1 = 2\tau[(2PQP-P)^{2j+1}] 
%\end{equation*}
%In a nutshell, 

When $P$ and $Q$ are free in $(\mathscr{A}, \tau)$, \eqref{Equa3} follows from the known property that the compression of $Q$ by $P$ has the same distribution in $(P\mathscr{A}P, 2\tau)$ as $(Q_1+Q_2)/2$ in $(\mathscr{A}, \tau)$. 
More generally, we do not know whether this property would imply freeness of $P$ and $Q$ at least in the special case $\tau(P) = 1/2$.

\end{document}